\documentclass[reqno,11pt]{amsart}
\usepackage{amsthm, amsmath, amssymb, stmaryrd}

\usepackage[utf8]{inputenc}
\usepackage[T1]{fontenc}
\usepackage{tikz}
\usepackage{ifthen}
\usetikzlibrary{decorations.markings}
\usetikzlibrary{calc}
\tikzset{vtx/.style={circle,fill=black,inner sep=1.0pt},
         e/.style={line width=0.6pt},
 vtxx/.style={circle,draw,minimum size=8.8mm,inner sep=0pt,line width=0.9pt},
  num/.style={circle,draw,minimum size=8.2mm,inner sep=0pt,line width=0.9pt},
  blackedge/.style={line width=0.9pt},
  rededge/.style={line width=1.6pt,red},
  elab/.style={red,above}       
         }
\usetikzlibrary{decorations.text}
\usetikzlibrary{arrows.meta} 

\usepackage{array,booktabs}
\newcolumntype{L}[1]{>{\raggedright\arraybackslash}m{#1}}
\newcolumntype{C}[1]{>{\centering\arraybackslash}m{#1}}

\usepackage{systeme}
\usepackage[shortlabels]{enumitem}
\usepackage{microtype}

\usepackage{bm}
\usepackage[margin=1in]{geometry}

\usepackage[textsize=scriptsize,backgroundcolor=orange!5]{todonotes}

\usepackage[hidelinks]{hyperref}
\usepackage[noabbrev,capitalize,sort]{cleveref}
\crefname{equation}{}{}
\crefname{enumi}{}{}
\numberwithin{equation}{section}

\usepackage{mathtools}

\newtheorem{theorem}{Theorem}[section]
\newtheorem{proposition}[theorem]{Proposition}
\newtheorem{lemma}[theorem]{Lemma}
\newtheorem{claim}[theorem]{Claim}
\newtheorem{corollary}[theorem]{Corollary}
\newtheorem{conjecture}[theorem]{Conjecture}

\theoremstyle{definition}
\newtheorem{definition}[theorem]{Definition}

\newtheorem{question}[theorem]{Question}
\newtheorem{example}[theorem]{Example}

\theoremstyle{remark}
\newtheorem*{remark}{Remark}

\newcommand{\abs}[1]{\left\lvert#1\right\rvert}

\newcommand{\ang}[1]{\left\langle #1 \right\rangle}

\DeclareMathOperator{\sgn}{sgn}
\DeclareMathOperator{\gr}{gr}
\DeclareMathOperator{\ind}{ind}
\DeclareMathOperator{\eind}{eind}
\DeclareMathOperator{\aind}{aind}
\DeclareMathOperator{\ldg}{ldg}
\DeclareMathOperator{\Aut}{Aut}

\DeclareMathOperator{\supp}{supp}

\newcommand*{\eqdef}{\stackrel{\mbox{\normalfont\tiny def}}{=}}

\newcommand{\RR}{\mathbb{R}}
\newcommand{\NN}{\mathbb{N}}
\newcommand{\ZZ}{\mathbb{Z}}
\newcommand*{\PP}{\mathbb{P}}

\newcommand{\cH}{\mathcal{H}}

\newcommand{\HH}{\mathbb{H}}

\newcommand{\Gedge}{
\begin{tikzpicture}[baseline=-0.6ex,scale=0.6]
  \node[vtx](a) at (0,0){}; \node[vtx](b) at (1.0,0){}; \draw[e](a)--(b);
\end{tikzpicture}}

\newcommand{\Gedgeedge}{
\begin{tikzpicture}[baseline=-0.6ex,scale=0.6]
  \node[vtx](a) at (0,0){}; \node[vtx](b) at (1,0){};
  \node[vtx](c) at (1,1){}; \node[vtx](d) at (0,1){};
  \draw[e](a)--(d);
  \draw[e](b)--(c);
\end{tikzpicture}}

\newcommand{\GstarKoneThree}{
\begin{tikzpicture}[baseline=-0.6ex,scale=0.6]
  \node[vtx](c) at (0,0){};
  \foreach \ang in {90,210,330}{
    \node[vtx](x\ang) at (\ang:0.9){};
    \draw[e](c)--(x\ang);
  }
\end{tikzpicture}}

\newcommand{\GpThree}{
\begin{tikzpicture}[baseline=-0.6ex,scale=0.6]
  \node[vtx](p1) at (0,0){}; \node[vtx](p2) at (0.9,0){}; \node[vtx](p3) at (1.8,0){};
  \draw[e](p1)--(p2)--(p3);
\end{tikzpicture}}

\newcommand{\GcFour}{
\begin{tikzpicture}[baseline=-0.6ex,scale=0.6]
  \node[vtx](a) at (0,0){}; \node[vtx](b) at (1,0){};
  \node[vtx](c) at (1,1){}; \node[vtx](d) at (0,1){};
  \draw[e](a)--(b)--(c)--(d)--(a);
\end{tikzpicture}}

\newcommand{\GkThree}{
\begin{tikzpicture}[baseline=-0.6ex,scale=0.6]
  \node[vtx](a) at (0,0){}; \node[vtx](b) at (1,0){};
  \node[vtx](c) at (0.5,0.866){};
  \draw[e](a)--(b)--(c)--(a);
\end{tikzpicture}}

\newcommand{\GkFour}{
\begin{tikzpicture}[baseline=-0.6ex,scale=0.6]
  \node[vtx](a) at (0,0){}; \node[vtx](b) at (1,0){};
  \node[vtx](c) at (1,1){}; \node[vtx](d) at (0,1){};
  \draw[e](a)--(b) (b)--(c) (c)--(d) (d)--(a);
  \draw[e](a)--(c) (b)--(d);
\end{tikzpicture}}

\newcommand{\GkFourMinusE}{
\begin{tikzpicture}[baseline=-0.6ex,scale=0.6]
  \node[vtx](a) at (0,0){}; \node[vtx](b) at (1,0){};
  \node[vtx](c) at (1,1){}; \node[vtx](d) at (0,1){};
  \draw[e](c)--(d) (b)--(c) (d)--(a);
  \draw[e](a)--(c) (b)--(d);
\end{tikzpicture}}

\newcommand{\GkThreePlusE}{
\begin{tikzpicture}[baseline=-0.6ex,scale=0.6]
  \node[vtx](a) at (0,0){}; \node[vtx](b) at (1,0){};
  \node[vtx](c) at (1,1){}; \node[vtx](d) at (0,1){};
  \draw[e](a)--(d)--(b)--(c)--(d);
\end{tikzpicture}}

\newcommand{\GpFour}{
\begin{tikzpicture}[baseline=-0.6ex,scale=0.6]
  \node[vtx](a) at (0,0){}; \node[vtx](b) at (1,0){};
  \node[vtx](c) at (1,1){}; \node[vtx](d) at (0,1){};
  \draw[e](a)--(d)--(c)--(b);
\end{tikzpicture}}

\newcommand{\GkFiveFiveSplit}{%
\begin{tikzpicture}[scale=0.6, baseline=-0.6ex]

\node[tinyv] (X1) at ( 0.00, 0.58) {}; 
\node[tinyv] (X2) at ( 0.55, 0.18) {}; 
\node[tinyv] (X3) at ( 0.34, -0.50) {}; 
\node[tinyv] (X4) at (-0.34, -0.50) {}; 
\node[tinyv] (X5) at (-0.55, 0.18) {}; 
\node[tinyv] (X6) at (-0.55, 0.18) {};
  \foreach \a/\b in {1/2,1/3,1/4,1/5,2/3,2/4,2/5,3/4,3/5,4/5}{
    \draw[tinyedge] (X\a) -- (X\b);
  }

   \node[tinyv] (Y1) at (1.8, 0.84) {};
  \node[tinyv] (Y2) at (1.8,  0.56) {};
  \node[tinyv] (Y3) at (1.8,  0.28) {};
  \node[tinyv] (Y4) at (1.8,  0.00) {};
  \node[tinyv] (Y5) at (1.8, -0.28) {};
  \node[tinyv] (Y6) at (1.8, -0.56) {};
  \node[tinyv] (Y7) at (1.8, -0.84) {};
  \foreach \i in {1,...,5}{
    \foreach \j in {1,...,7}{
      \draw[faintedge] (X\i) -- (Y\j);
    }
  }
\end{tikzpicture}
}

\newcommand{\GkSixWithSixPendants}{%
\begin{tikzpicture}[scale=0.6, baseline=-0.6ex]
  \foreach \k [count=\i] in {90,150,210,270,330,30}{
    \node[tinyv] (c\i) at ({0.9*cos(\k)},{0.9*sin(\k)}) {};
  }
  \foreach \i in {1,...,6}{\foreach \j in {1,...,6}{\ifnum\j>\i \draw[faintedge] (c\i)--(c\j);\fi}}
  \def\delta{8} \def\R{1.55}
  \foreach \i/\ang in {1/90,2/150,3/210,4/270,5/330,6/30}{
    \foreach \s in {1,...,6}{
      \pgfmathsetmacro{\th}{\ang + (\s-3.5)*\delta}
      \node[tinyv] (p\i\s) at ({\R*cos(\th)},{\R*sin(\th)}) {};
      \draw[tinyedge] (c\i)--(p\i\s);
    }
  }
\end{tikzpicture}
}

\tikzset{
  tinyv/.style={circle, fill=black, inner sep=0.8pt},
  tinyedge/.style={line width=0.28pt},
  faintedge/.style={line width=0.26pt, opacity=0.35},
}
\newcommand{\drawComplete}[1]{%
  \foreach \i [count=\ii] in {#1} {%
    \foreach \j [count=\jj] in {#1} {%
      \ifnum\jj>\ii \draw[tinyedge] \i -- \j; \fi
    }%
  }%
}

\newcommand{\GkSix}{
\begin{tikzpicture}[scale=0.6, baseline=-0.6ex]
  \foreach \k [count=\i] in {90,150,210,270,330,30}{
    \node[tinyv] (v\i) at ({0.9*cos(\k)},{0.9*sin(\k)}) {};
  }
  \foreach \i in {1,...,6}{\foreach \j in {1,...,6}{\ifnum\j>\i \draw[faintedge] (v\i)--(v\j);\fi}}
\end{tikzpicture}
}

\newcommand{\GmatchingFour}{%
\begin{tikzpicture}[scale=0.6, baseline=-0.6ex]
  \foreach \i [count=\r] in {1,...,4} {
    \node[tinyv] (a\r) at (0,1.2-0.4*\r) {};
    \node[tinyv] (b\r) at (1.1,1.2-0.4*\r) {};
    \draw[tinyedge] (a\r)--(b\r);
  }
\end{tikzpicture}
}

\newcommand{\GkFourFour}{
\begin{tikzpicture}[scale=0.6, baseline=-0.6ex]
  \foreach \i [count=\r] in {1,...,4} { \node[tinyv] (L\r) at (0,0.6-0.4*\r) {}; }
  \foreach \j [count=\c] in {1,...,4} { \node[tinyv] (R\c) at (1.6,0.6-0.4*\c) {}; }
  \foreach \r in {1,...,4}{\foreach \c in {1,...,4}{\draw[faintedge] (L\r) -- (R\c);}}
\end{tikzpicture}
}

\newcommand{\GstarKoneSix}{
\begin{tikzpicture}[scale=0.6, baseline=-0.6ex]
  \node[tinyv] (c) at (0,0) {};
  \foreach \k in {0,60,120,180,240,300}{
    \node[tinyv] (l\k) at ({1.0*cos(\k)},{1.0*sin(\k)}) {};
    \draw[tinyedge] (c) -- (l\k);
  }
\end{tikzpicture}
}

\newcommand{\GTrianglePendant}{\scalebox{0.8}{\begin{tikzpicture}[scale = 0.4,
  plusmarks/.style={
    postaction=decorate,
    decoration = {
            markings,
            mark = 
                between positions 0.08 and 0.92 step 6pt 
                with
                {
                    \draw (0pt, 2pt) -- (0pt, -2pt);
                    \draw (-2pt, 0pt) -- (2pt, 0pt);
                }
            }
  },
  halfplus/.style={
    postaction=decorate,
    decoration={
            markings,
            mark = 
                between positions 0.08 and 0.5 step 6pt 
                with
                {
                    \draw (0pt, 2pt) -- (0pt, -2pt);
                    \draw (-2pt, 0pt) -- (2pt, 0pt);
                }
    }
  },
  halfminus/.style={
    postaction=decorate,
    decoration={markings,
      mark=between positions 0.5 and 0.92 step 6pt with
        {
                    \draw (-2pt, 0pt) -- (2pt, 0pt);}
    }
  }
]

\coordinate (A) at (0,0);
\coordinate (B) at (5.1,0);
\coordinate (C) at (10.2,0);
\draw [fill] (A) circle (5pt);
\draw [fill] (B) circle (5pt);
\draw [fill] (C) circle (5pt);

\path[plusmarks] (0,-0.3)  --  (5.1,-0.3);
\path[halfplus]  (0,0.3)  --  (5.1,0.3);
\path[halfminus] (0,0.3)  --  (5.1,0.3);

\path[plusmarks] (5.1,-0.3)  --  (10.2,-0.3);
\path[halfplus]  (5.1,0.3)  --  (10.2,0.3);
\path[halfminus] (5.1,0.3)  --  (10.2,0.3);

\path[plusmarks] (0.3-0.6,0.25)  .. controls (5.1,2.7) ..  (9.8+0.6,0.25);
\path[halfplus]  (-0.3-0.5,0.3)  .. controls (5.1,3.3) ..   (10.5+0.5,0.3);
\path[halfminus] (-0.3-0.5,0.3)  .. controls (5.1,3.3) ..   (10.5+0.5,0.3);

\end{tikzpicture}}
}

\newlength{\hght}

\makeatletter
\newcommand\thankssymb[1]{\textsuperscript{\@fnsymbol{#1}}}
\makeatother

\author[Ting-Wei Chao]{Ting-Wei Chao\thankssymb{1}}

\author[Asaf Cohen Antonir]{Asaf Cohen Antonir\thankssymb{3}}

\author[Anqi Li]{Anqi Li\thankssymb{4}}

\author[Hung-Hsun Hans Yu]{Hung-Hsun Hans Yu\thankssymb{2}}

\thanks{\thankssymb{1}Department of Mathematics, Massachusetts Institute of Technology, Cambridge, MA 02139, USA. Email: {\tt twchao@mit.edu}}

\thanks{\thankssymb{2}Department of Mathematics, Princeton University, Princeton, NJ 08544\@.  Email: {\tt hansonyu@princeton.edu}}

\thanks{\thankssymb{3}School of Mathematical Sciences, Tel Aviv University, Tel Aviv 69978, Israel. Email {\tt asafc1@tauex.tau.ac.il}}

\thanks{\thankssymb{4}Department of Mathematics, Stanford University, Stanford, CA 94305, USA. Email: {\tt aqli@stanford.edu}}

\title{Edge inducibility via local directed graphs}
\begin{document}

\begin{abstract}
    In this paper we introduce the \emph{edge inducibility} problem. This is a common refinement of both the well known Kruskal--Katona theorem and the inducibility question introduced by Pippenger and Golumbic. 

    Our first result is a hardness result.
    It shows that for any graph $G$, there is a related graph $G'$ whose edge inducibility determines the vertex inducibility of $G$.
    Moreover, we determine the edge inducibility of every $G$ with at most $4$ vertices, and make some progress on the cases $G=C_5,P_6$. Lastly, we extend our hardness result to graphs with a perfect matching that is the unique fractional perfect matching. This is done by introducing \emph{locally directed graphs}, which are natural generalizations of directed graphs. 
\end{abstract}

\maketitle

\section{Introduction}

Given a graph $G$, determining the maximum number of copies or induced copies of $G$ in a graph $H$ with some combinatorial restrictions lies at the heart of extremal graph theory. Two well-known problems falling under this description are the Kruskal--Katona theorem and the inducibility question. In this paper, we study a new and natural question in the common refinement of these problems.

To make our discussion rigorous, for graphs $G$ and $H$, set $N(G, H)$ and $N_{\ind}(G, H)$ to be the number of copies and induced copies of $G$ in $H$, respectively. The above question can then be restated, asking to determine 
\[
	c_{\mathcal{H}}(G)\eqdef \sup \{N(G,H) : H\in \mathcal{H} \} \quad \text{and}\quad \ind_{\mathcal{H}}(G)\eqdef \sup\{N_{\ind}(G,H) : H\in \mathcal{H} \},
\]
where $\mathcal{H}$ is some family of graphs.

The inducibility question of Pippenger and Golumbic \cite{PipGol1975} asks to determine $\ind_{\mathcal{H}}(G)$ for the most natural family of graphs $\mathcal{H}$, the family of $n$-vertex graphs. Writing $\ind(G,n) = \ind_{\mathcal{H}_n}(G)$ for $\mathcal{H}_n$ the family of $n$-vertex graphs, it was shown by Pippenger and Golumbic that there is a constant $\ind(G)\in [0,1]$, called the \textbf{inducibility} of $G$, such that 
\[
	\lim_{n\to \infty} \frac{\ind(G,n)}{\binom{n}{v(G)}} \eqqcolon \ind(G).
\]  
Despite the significant effort in determining $\ind(G)$ for various graphs $G$ (see \cite{BalHuLidPfe2016,BonPik2025, BolEgaHarJin1995, EveLin2015, HatHirNor2014, HefTyo2018,Hir2014, KraNorVol2019, MorSco2017, PikSliTyr2019, Uel2024, Yus2019}), even the case of $G=P_4$, the path on $4$ vertices, remains unsolved \cite{EveLin2015}. As was observed by Pippenger and Golumbic \cite{PipGol1975}, for a $k$-vertex graph $G$, an iterated blowup\footnote{An iterated blowup of a graph $G$ is a graph that can be obtained by repeatedly blowing up each vertex with $\abs{V(G)}$ copies of itself that induce a graph isomorphic to $G$.} gives rise to a lower bound $\ind(G)\geq \frac{k!}{k^k-k}$. They conjectured that for $k\geq 5$, this lower bound is tight for $C_k$, the cycle on $k$ vertices. There were several attempts to resolve this conjecture for large $k$ \cite{HefTyo2018,MorSco2017,PipGol1975}, where the current state of the art is due to Kr\'al', Norin, and Volec \cite{KraNorVol2019}, who showed an upper bound far by a factor of $2$ from the conjectured lower bound. As for small values of $k$, this conjecture is solved only when $k=5$ by Balogh, Hu, Lidick\'y, and Pfender \cite{BalHuLidPfe2016}, leaving the rest of the cases wide open. The method in \cite{BalHuLidPfe2016} uses a stability result together with the famous flag algebra method of Razbarov \cite{Raz2007}. This method was also used to determine the inducibility of other small graphs (see \cite{BonPik2025,EveLin2015,Hir2014,PikSliTyr2019}). 
Lastly, let us remark that for large blowups \cite{HatHirNor2014} and ``quasirandom'' graphs \cite{FoxSauWei2021,Yus2019} it was shown that the iterated blowups give rise to optimal bounds.

Let us now turn our attention to the parameter $c_{\mathcal{H}}(G)$. Note that if we take $\mathcal{H}$ to be $\mathcal{H}_n$, the quantity $c_{\mathcal{H}}(G)$ is trivially maximized by $H=K_n$. The next natural candidate to examine is the family of graphs having $m$ edges. 
For simplicity, let us write $c(G,m)\eqdef c_{\mathcal{H}}(G)$ with this choice of $\cH$. 
The celebrated Kruskal--Katona theorem \cite{Kat68,Kru63} provides an exact formula for $c(K_r,m)$ for every nonnegative integers $r,m$.
Lov\'asz also proved a weaker but cleaner version of the Kruskal--Katona theorem, namely that $c\left(K_r,\binom{x}{2}\right)\leq \binom{x}{r}$ for any real number $x\geq r$.
The interested reader is referred to \cite[Chapter 13, Exercise 31b]{Lovasz93}, as well as a new entropic proof of this theorem \cite{ChaYu2024KK}. 

In the general case, Alon \cite{Alo1981} determined, for each fixed $G$, the function $c(G,m)$ up to a multiplicative factor. 
This was later reiterated and generalized to hypergraphs by Friedgut and Kahn \cite{FriKah1998} via a simple proof using duality of linear programs. Their proof can be stated beautifully via entropy as was done in the excellent survey of Galvin \cite{Gal2014}.

Recalling that $\alpha^*(G)$ denotes the \emph{fractional independence number} of $G$ (see Section \ref{sec:prelim} for the definition), the results of Alon and of Friedgut and Kahn can be stated as follows.

\begin{theorem}[Alon \cite{Alo1981}, Friegdut--Kahn \cite{FriKah1998}]\label{thm: Alon}
For a graph $G$ without isolated vertices and $m\in \mathbb{N}$
	\[ 
		(1-o_G(1))\left(\frac{m}{e(G)}\right)^{\alpha^*(G)}\le c(G,m)\le \frac{(2m)^{\alpha^*(G)}}{|\Aut(G)|}.
	\]
\end{theorem}

Besides the elegance in the statement and its proofs, Theorem \ref{thm: Alon} became central in studying a seemingly unrelated major problem in random graphs---The Infamous Upper Tail \cite{JanRuc2002}. 
Indeed, in the influential work on the upper tail problem by Janson, Oleszkiewicz, and Ruci\'nski \cite{JanOleRuc2004}, evaluating $c_\mathcal{H}(G)$ for the family $\mathcal{H}$ of $n$-vertex and $m$-edge graphs, was key. 
This was later used crucially in the recent breakthrough of Harel, Mousset, and Samotij \cite{HarMouSam2022}.

Finally, our attention is directed to the new problem that we suggest. Indeed, the above discussion naturally leads to the following question.
\begin{question}\label{question: edge inducibility}
	 For a graph $G$ and a nonnegative integer $m$, setting $\mathcal{H}$ to the family of $m$-edge graphs, what is the asymptotic value of $\eind(G,m)\eqdef \ind_{\mathcal{H}}(G)$?
\end{question}

In this paper, we are interested in studying Question \ref{question: edge inducibility} when $m$ is large. Our first observation is that for every graph $G$ without isolated vertices and a nonnegative integer $m$, we have
\begin{equation}\label{eq: trivial bounds on edge inducibility}
    (1-o_G(1))\left(\frac{m}{e(G)}\right)^{\alpha^*(G)}\le \eind(G,m)\le \frac{(2m)^{\alpha^*(G)}}{\abs{\Aut(G)}}.
\end{equation}
This is an immediate corollary of Theorem \ref{thm: Alon}, as shown in Lemma \ref{lem:simple bound}. Given this observation, it is plausible to make the following central definition of this paper.
\begin{definition}
	For a graph $G$ without isolated vertices, define its \textbf{edge inducibility} to be
\[
	\eind(G) \eqdef \limsup_{m\to\infty} \frac{\abs{\Aut(G)}\eind(G,m)}{(2m)^{\alpha^*(G)}}\in [0,1].
\]
\end{definition}

We remark that this graph parameter was implicitly studied by  Bollob\'as, Nara, and Tachibana \cite{BolNarTac1986} for the case of balanced complete bipartite graphs. 

Our first result, perhaps surprising, shows that the edge inducibility problem---i.e.\ determining $\eind(H)$ for a fixed graph $H$---is ``harder'' than the classical vertex inducibility. Formally, this is given by the following theorem.
\begin{theorem}\label{thm:hardness-result}
	Suppose that $G$ is a graph, then there is a graph $G'$ such that
	\[
		 \frac{\ind(G)}{2^{v(G)} v(G)! }=\frac{\eind(G')}{|\Aut(G')|}.
	\]
\end{theorem}

Let us note that this is indeed a hardness-type result, as resolving the edge inducibility yields a solution to the vertex inducibility.
That said, for natural graphs, this problem seems more tractable than the vertex inducibility. Supporting this, in \cref{section: graphs smaller than 4} we compute the edge inducibility for all $4$-vertex graphs, a task that is still out of reach for the vertex inducibility \cite{EveLin2015,Hir2014}. Moreover, our proofs are non-computer-assisted, in contrast to many proofs in the vertex inducibility realm \cite{BonPik2025,EveLin2015,Hir2014,PikSliTyr2019}. 

In fact, our argument for \cref{thm:hardness-result} can be generalized a lot, which turns out to be a useful machinery for bounding edge inducibilities of certain graphs.
To be more specific, the graphs where the argument can be generalized to are precisely graphs $G$ with a perfect matching $M$ that is the unique \emph{fractional perfect matching} (see \cref{sec:prelim} for the definition).
For any graph $G$ that has this property, it turns out that its edge inducibility is in close relation to the vertex inducibility of a ``locally directed graph'' that $G$ induces.
Due to its technicality, we only remark that a locally directed graph is a generalization of a directed graph, and we postpone its formal definition until \cref{section: local digraph}.
We will also not state the generalization of \cref{thm:hardness-result} until then, and instead we list several consequences of the generalization here.

The first is a lower bound on $\eind(G)$ that is better than what \cref{eq: trivial bounds on edge inducibility} gives.
Equation \cref{eq: trivial bounds on edge inducibility} gives that $\eind(G)\geq \frac{\abs{\Aut(G)}}{(2e(G))^{\alpha^*(G)}}$ in general.
However, if $G$ admits a perfect matching that is the unique fractional perfect matching, we have the following better lower bound.

\begin{theorem}\label{thm:better-construction-intro}
    Let $G$ be a graph with a perfect matching $M$ of size $k$ that is the unique fractional perfect matching.
    Then
    \[\eind(G)\geq \frac{\abs{\Aut(G)}}{2^{k}(k^{k}-k)}.\]
\end{theorem}
Note that in this case, $\alpha^*(G)=k$ and $e(G)\geq k$, showing that this is indeed always an improvement whenever the theorem applies.

In addition, for concrete graphs that satisfy the assumption, we may apply the generalization of \cref{thm:hardness-result} to have a fairly good control on the edge inducibility.
Instead of stating the most general result, we state some key examples that we deal with in the paper.

\begin{theorem}\label{thm:local-digraph-small-case}
    Let $K_3^+$ be the graph obtained by attaching an edge to one of the vertices of a triangle $K_3$.
    Let $P_k$ be a path on $k$ vertices.
    Then
    \[\eind(K_3^+)=\eind(P_4)=\frac{1}{4}\]
    and
    \[\frac{5}{372}\leq \eind(P_6)\leq \frac{1}{36}.\]
\end{theorem}

In addition to small graphs and paths, we also consider cycles in this paper.
The bound in \eqref{eq: trivial bounds on edge inducibility} gives $\eind(C_k) \ge 2k\left(\frac{1}{2k}\right)^{k/2}$, which comes from counting induced copies in the balanced blowup of $C_k$. 
We conjecture that this lower bound is in fact tight whenever $k\geq 4$. 
\begin{conjecture}\label{conj:cycle}
	For every integer $k\geq 4$, we have $\eind(C_k) = 2k\left(\frac{1}{2k}\right)^{k/2}$.
\end{conjecture}

Towards this, in Section \ref{section: graphs smaller than 4} we confirm the conjecture for $k=4$ (which was also done implicitly in both \cite{BolNarTac1986,Coh2024}). Moreover, using the entropy method in a delicate way, inspired by the recent work of the first and last author \cite{CY24turan}, we prove the following bound in the case of $k=5$, which is far by a multiplicative factor of $2$ from the conjectured lower bound $\frac{1}{\sqrt{1000}}$.

\begin{theorem}\label{thm:C5}
	We have $\eind(C_5)\leq \frac{1}{\sqrt{250}}$.  
\end{theorem}

We remark that there is a recent independent work by Wang, Zhao and Lu \cite{WanZhaLu2025} that also poses and studies the edge inducibility problem.
Our \cref{conj:cycle} coincides with their Conjecture 1.7, and \cref{thm:local-digraph-small-case} disproves their Conjecture 1.8 for $k=4,6$.
We remark that \cref{thm:better-construction-intro} can be used to disprove their Conjecture 1.8 for all $k\geq 4$, and we make this remark after its proof in \cref{section: local digraph}.

\subsection{Outline of the paper}
We begin in \cref{sec:prelim} by introducing the necessary notation and the Shannon entropy, both of which will be used frequently throughout the paper. We also prove \cref{eq: trivial bounds on edge inducibility} in this section.
Then in \cref{sec: edge ind is hard}, we state a more precise version of \cref{thm:hardness-result} and give a proof.
In \cref{section: local digraph}, we first define locally directed graphs and what it means for them to be acyclic.
We then use the built terminology and basic properties to state and prove the generalization of \cref{thm:hardness-result}, which allows us to prove \cref{thm:better-construction-intro} and also understand the edge inducibility of $P_6$.
In the following section, we include our results about small graphs, including all graphs on at most $4$ vertices and also the cycle of length $5$.
Finally, we close with some remarks and conjectures in \cref{sec:conclusion}.

\section{Preliminaries}\label{sec:prelim}
\subsection{Notations}
Throughout the paper, we use the following standard graph theoretic notations.  For a graph $G$, we write $V(G)$ to denote its vertex set, and write $v(G)$ to denote its number of vertices. Furthermore, we write $E(G)$ for the edge set of $G$ and write $e(G)$ to denote the size of $E(G)$. 

We also use the following standard notions from fractional graph theory. 
Suppose that $G$ is a graph. A function $\alpha\colon V(G)\to [0,1]$ is said to be a \emph{fractional independent set} of $G$ if for every $uv\in E(G)$ we have $\alpha(u)+\alpha(v)\le 1$.
The \emph{fractional independence number} of $G$, denoted by $\alpha^*(G)$, is the maximum of $\sum_{v\in V(G)}\alpha(v)$, running over all fractional independent sets $\alpha$ of $G$.
Lastly, a function $w\colon E(G)\to [0,1]$ is called a \emph{fractional matching} of $G$ if for every $v\in V(G)$, we have $\sum_{u:uv\in E(G)}w(uv)\leq 1$.
If the equality holds at every vertex, then we call it a \emph{fractional perfect matching}.

Finally, we will also use some notations regarding homomorphisms and embeddings when it is more convenient for us.
Let $G,H$ be two graphs.
We say that a function $\varphi\colon V(G)\to V(H)$ is a \emph{homomorphism} from $G$ to $H$ if for every $uv\in E(G)$, we have $\varphi (uv)\eqdef \varphi(u)\varphi(v)\in E(H)$. 
An injective homomorphism $\varphi$ is called an \emph{embedding}. 
If an embedding $\varphi$ from $G$ to $H$ furthermore satisfies $\varphi(uv)\not\in E(H)$ for every $uv\not\in E(G)$, we call $\varphi$ an \emph{induced embedding}. 
We denote by $E_{\ind}(G,H)$ the number of induced embeddings from $G$ to $H$, and note that $E_{\ind}(G,H)=\abs{\Aut(G)} \cdot N_{\ind}(G,H)$.
Let us remark that with this notation, for every graph $G$ we have
\[
	\eind(G) = \limsup_{m\to\infty} \max_{H: e(H)=m}\frac{E_{\ind}(G,H)}{(2m)^{\alpha^*(G)}},
\]
which we sometime use.

\subsection{Shannon entropy}
For any discrete random variable $X$, its support $\supp(X)$ is the set $\{x\mid \PP(X=x)>0\}$. Throughout this paper, every random variable mentioned is discrete and has a finite support, i.e.\ the size of its support is finite. Now, we may introduce the definition of Shannon entropy of a random variable.

\begin{definition}
    Given a random variable $X$, its Shannon entropy is given by
    \[\HH(X)\eqdef \sum_{x\in\supp(X)}-\PP(X=x)\log_2\PP(X=x).\]
    For any sequence of random variables $X_1,\dots,X_k$, the joint entropy $\HH(X_1,\dots,X_k)$ is the entropy of the random tuple $(X_1,\dots,X_k)$. 

    We also define the conditional entropy of $X$ given $Y$, 
    \[\HH(X\mid Y)\eqdef \HH(X,Y)-\HH(Y).\]
\end{definition}

Next, we state some properties that are frequently used in \cref{sec: five cycle}. We refer the readers to \cite{AS00,Gal2014} for more details about entropy. First, we have the uniform bound which helps us connect the number of objects we are counting and the entropy of random variables.

\begin{proposition}[Uniform bound]
    For any random variable $X$, we have
    \[\HH(X)\leq \log_2\abs{\supp(X)},\]
    where equality holds if and only if $X$ is a uniform random variable on $\supp(X)$.
\end{proposition}

Next, we have the chain rule, which follows from the telescoping sum of conditional entropy.
\begin{proposition}[Chain rule]
    For any random variables $X_1,\dots,X_k$, we have
    \[\HH(X_1,\dots,X_k)= \HH(X_1)+\HH(X_2\mid X_1)+\dots+\HH(X_k\mid X_1,\dots,X_{k-1}).\]
\end{proposition}

We also need the following two inequalities about entropy.
\begin{proposition}[Dropping conditioning]
    For any random variables $X,Y,Z$, we have
    \[\HH(X\mid Y,Z)\leq\HH(X\mid Y).\]
\end{proposition}
\begin{proposition}[Subadditivity]
    For any random variables $X,Y,Z$, we have
    \[\HH(X,Y\mid Z)\leq\HH(X\mid Z)+\HH(Y\mid Z).\]
\end{proposition}

The last proposition we need is the mixture bound. This bound plays an important role in the recent development of proving the Kruskal--Katona theorem and Tur\'an's theorem with the entropy method \cite{CDSY24,ChaYu2024KK,CY24turan}. Here, we only state a simplified version that we need. The full statement and proof can be found in \cite{CY24turan}.

\begin{definition}[Mixture]
    Given random variables $X_1,\dots,X_k$, we say that $Z$ is a mixture of $X_1,\dots,X_k$ if $Z$ is of the form $X_{\mathbf{i}}$, where $\mathbf{i}\in [k]$ is an independent random index. 
\end{definition}
\begin{proposition}[Mixture bound]\label{prop:mix}
    Let $X_1,\dots,X_k$ be random variables with mutually disjoint supports. That is, $\supp(X_i)\cap\supp(X_j)=\varnothing$ holds for any distinct $i,j\in [k]$. Then there exists a mixture $Z$ of $X_1,\dots,X_k$ such that
    \[2^{\HH(X_1)}+\dots+2^{\HH(X_k)}\leq 2^{\HH(Z)}.\]
\end{proposition}

\subsection{Simple bounds on the edge inducibility}

We conclude this section with the following lemma, which as mentioned in the introduction, is a simple corollary of Theorem \ref{thm: Alon}.

\begin{lemma}\label{lem:simple bound}
    For every graph $G$ without isolated vertices and a nonnegative integer $m$ we have
    \[
        (1-o_G(1))\left(\frac{m}{e(G)}\right)^{\alpha^*(G)}\le \eind(G,m)\le \frac{(2m)^{\alpha^*(G)}}{\abs{\Aut(G)}}.
    \]
\end{lemma}

\begin{proof}
    Note that $\eind(G,m)\leq c(G,m)$ always, and thus Theorem \ref{thm: Alon} implies the required upper bound.
    For the lower bound, let $\alpha\colon V(G) \to [0,1]$ be a fractional independence set achieving $\alpha^*(G)$. Let $(V_v:v\in V(G))$ be a sequence of mutually disjoint sets satisfying $\abs{V_v}= \left\lfloor\left({m}/{e(G)}\right)^{\alpha(v)}\right\rfloor$. Define $G^*$ to be the graph with vertex set $\bigcup_{v\in V(G)} V_v$ and edge set consisting of all pairs $xy$, where $x\in V_v$ and $y\in V_u$ such that $uv\in G$. Note that 
    \[
        e(G^*) = \sum_{uv\in E(G)} \abs{V_v}\abs{V_u} \le \sum_{uv\in E(G)} \left(\frac{m}{e(G)}\right)^{\alpha(v)+\alpha(u)} \leq m.
    \]
    Moreover, 
    \[
        N_{\ind}(G,G^*) \geq \prod_{v\in V(G)} \abs{V_v} =(1-o_G(1))  \left(\frac{m}{e(G)}\right)^{\alpha^*(H)}.
    \]
    This implies the required lower bound.
\end{proof}

By definition, we immediately have the following corollary.

\begin{corollary}\label{cor:simple bound}
    For every graph $G$ without isolated vertices,
    \[\frac{\abs{\Aut(G)}}{(2e(G))^{\alpha^*(G)}}\leq \eind(G,M)\leq 1.\]
\end{corollary}

\section{Edge inducibility is harder than vertex inducibility}\label{sec: edge ind is hard}

The starting point for our discussion is the edge inducibility problem for $P_4$, which is notably the only 4-vertex graph for which vertex inducibility in the sense of Pippenger and Golumbic \cite{PipGol1975} is not completely understood. We think of $P_4$ as the union of a matching and an edge; for each of these components (matching and edge) we know the extremal host graph construction (matching and complete graph, respectively). A natural question would be whether we can suitably ``stitch together'' these two constructions into an extremal host graph for $P_4$.

\begin{example}[Edge inducibility for $P_4$]\label{ex:P4}
    One natural construction for the host graph $H$ is to start with a complete graph $K_t$, for $t$ to be chosen, and then attach an $s$-star to each vertex of the $K_t$ (see \cref{fig:P4-host}). Sometimes this graph can be written using the Corona product \cite{FH70} as $K_t \circ \overline{K_s}$.
    In this case, $e(H) = st + \binom{t}{2}$. The number of embeddings of $P_4$ in $H$ can be counted by selecting an edge in the $K_t$ and then choosing a vertex in each of the stars at the endpoints, for a total of $2 \cdot \binom{t}{2} \cdot s^2$. By taking $t \ll s$, and letting $s,t \to \infty$, we can check that $\eind(P_4)\geq \frac{E_{\ind}(P_4,H)}{(2e(H))^2} \to \frac{1}{4}$. 
    
    On the other hand, we may also note that taking every alternate edge of $P_4$ gives a matching on two edges. We may easily upper bound the number of induced embeddings of $P_4$ in the host graph $H$ by the number of ordered pairs of edges, giving $E_{\ind}(P_4,H) \leq m^2-m$. That is, $\eind(P_4) \leq \lim_{m \to \infty} \frac{m^2 - m}{(2m)^2} = \frac{1}{4}$ showing that the lower bound is in fact tight and we have $\eind(P_4) = \frac{1}{4}$.

    \begin{figure}[!htp] \begin{tikzpicture}[
    scale=0.8,
    vertex/.style={circle, fill=black, inner sep=1.6pt},
    leaf/.style={circle, fill=red!70!black, inner sep=1.2pt},
    edge/.style={line width=0.8pt},
    leafedge/.style={draw=red, line width=0.8pt}
  ] \def\t{5}       
\def\s{4}       
\def\R{1.25}    
\def\Rleaf{2.05}
\def\spread{12}

\foreach \i in {1,...,\t}{
  \pgfmathsetmacro{\ang}{90 + 360*(\i-1)/\t}
  \node[vertex] (v\i) at (\ang:\R) {};
}

\foreach \i in {1,...,\t}{
  \pgfmathtruncatemacro{\next}{\i}
  \foreach \j in {\next,...,\t}{
    \draw[edge] (v\i) -- (v\j);
  }
}

\foreach \i in {1,...,\t}{
  \pgfmathsetmacro{\baseang}{90 + 360*(\i-1)/\t}
  \foreach \j in {1,...,\s}{
    \pgfmathsetmacro{\offset}{(\j-(\s+1)/2)*\spread}
    \node[leaf] (u\i-\j) at (\baseang+\offset:\Rleaf) {};
    \draw[leafedge] (v\i) -- (u\i-\j);
  }
}

\node at (270:\R+0.2) {$K_t$};
\node[text=red!70!black] at (90:\Rleaf+0.4) {$s$ vertices};

\end{tikzpicture}
\caption{Host graph for edge inducibility of $P_4$.}
\label{fig:P4-host}
\end{figure}
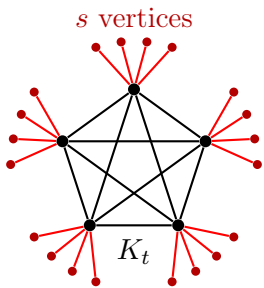
\end{example}

This example suggests why the edge inducibility problem for $P_4$ is more tractable than the vertex inducibility problem for $P_4$, since it is in fact in some sense equivalent to the vertex inducibility problem for an edge. A moment of thought reveals that this same kind of argument shows that for general graphs we expect the edge inducibility problem to be basically at least as difficult as the vertex inducibility problem. Indeed, for any graph $G$, construct $G'$ by attaching a pendant edge to each vertex of $G$. The next theorem shows that the edge inducibility problem for $G'$ can be reduced to the vertex inducibility problem of $G$. 

\begin{theorem}[Edge inducibility is ``as hard as'' vertex inducibility]\label{thm:eind-as-hard-as-ind}
    Let $G$ be a graph on $n$ vertices, and let $G'$ be obtained by attaching a pendant edge to each vertex of $G$. Then
\[ \frac{\ind(G)}{2^n n!}=\frac{ \eind(G')}{|\Aut(G')|} . \]
\end{theorem}

\begin{proof}
    For the upper bound on $\eind(G')$, the main idea is that ``pinching'' the pendant edges of $G'$ in the extremal host graph $H'$ gives $G$. As such, for every host graph $H_m'$ with $m$ edges, that satisfies $N_{\ind}(G',H'_m)=\eind(G',m)$, we construct an auxiliary host graph $H_m$ on $m$ vertices. This is done as follows: each edge $e$ of $H_m'$ corresponds to a vertex $v_e$ of $H_m$, and the edge set of $H_m$ is defined to be $v_ev_f$ for $e,f \in H_m'$ such that $e \cap f = \varnothing$, and there is an edge of $H'_m$ intersecting both $e$ and $f$. This way, every induced copy of $G'$ in $H'_m$ gives rise to an induced copy of $G$ in the host graph $H_m$ that we constructed. By accounting for automorphisms of $G'$ and the fact that $\alpha^*(G) = n$, we get
    \begin{align*}
        \eind(G') &=  \limsup_{m \to \infty}\frac{\abs{\Aut(G')}N_{\ind}(G',H_m')}{(2m)^n}\\
&\leq  \limsup_{m\to\infty}\left(\frac{n!\cdot N_{\ind}(G,H_m)}{m^n}\right)\cdot
\left(\frac{\abs{\Aut(G')} }{2^n n!}\right)\\
&=   \ind(G)\cdot \frac{\abs{\Aut(G')}}{2^{n}n!}.
    \end{align*}
    
    To prove the lower bound on $\eind(G')$, we draw inspiration from Example~\ref{ex:P4}. We start with the extremal host graph $H$ for the vertex inducibility of $G$, and add a star with $s$ vertices to each vertex of $H$ to form the host graph $H'$ for $G'$. More precisely, let $H$ be the extremal host graph for the vertex inducibility of $G$ on $t$ vertices. As in Example~\ref{ex:P4}, for some $s$ to be chosen we define $H' \eqdef H \circ \overline{K_s}$. Now, each induced copy of $G$ in $H$ corresponds to $s^{n}$ many induced copies of $G'$ in $H'$, and the total number of edges in $H'$ is bounded above by $e(H') \leq \binom{t}{2} + st$.  By taking $s \gg t$ so that $s$ grows sufficiently fast relative to $t$, and letting $t \to \infty$, we can bound the number of induced copies of $G'$ in $H'$ to get 
    \[ \mathrm{eind}(G') \geq \abs{\Aut(G')} \cdot \limsup_{t \to \infty}\frac{s^{n} \cdot\ind(G,t)}{\left(2 \cdot \left(\binom{t}{2} + st \right) \right)^{n}} = \frac{\abs{\Aut(G')} \cdot \ind(G)}{2^n n!}, \]
    as desired. 
\end{proof}

\section{Generalization: local digraph} \label{section: local digraph}

In the proof of \cref{thm:eind-as-hard-as-ind}, the strategy to relate the edge inducibility of some graph to the vertex inducibility of some other graph is the following.
We think of edges in a host graph $H$ as vertices in some other auxiliary graph $H'$. 
In addition, for any two edges $e,f\in E(H)$ with $e\cap f=\varnothing$, we may record whether there are any edges going across $e$ and $f$ using edges of $H'$.
However, this transformation going from $H$ to $H'$ is somewhat lossy, as there are actually four possible edges going between $e$ and $f$.
The motivation of considering locally directed graphs is to fully retain information from $H$ when transforming to $H'$ in the hope that we can generalize \cref{thm:eind-as-hard-as-ind} further.
It turns out that with the correct setup, we not only can generalize the upper bound on the edge inducibility in \cref{thm:eind-as-hard-as-ind}, but we can also extend the argument for the lower bound.

\subsection{Locally directed graphs}
In this subsection, we first define locally directed graphs and then define some operations that will be useful later.

A directed graph can be thought of as a graph where for each edge, it is an out-edge at one endpoint and an in-edge at the other endpoint.
If we allow an edge to be an out-edge at each endpoint, or an in-edge at each endpoint, we obtain what we call a \emph{locally directed graph}.
To avoid confusion and for our convenience, we will use plus and minus signs instead to specify the ``direction'' of an edge at a vertex.
\begin{definition}[Locally directed graph]\label{def:local-digraph}
    A \emph{locally directed graph}, or \emph{local digraph}, is a multigraph with no self-loops with the following additional data: for each vertex $v$ and each edge $e$ incident to it, the pair $(v,e)$ has either a positive sign or a negative sign.
    Let $\textup{sgn}(v,e)$ be the sign of $e$ at $v$.
    
    Two local digraphs $G,G'$ are thought to be isomorphic if there are bijections $f = (f_V,f_E)$ where $f_V\colon V(G)\to V(G')$ and $f_E\colon E(G)\to E(G')$ so that it is an isomorphism between multigraphs and also for any two edges $e,e'\in E(G)$ with some vertex $v\in e\cap e'$, we have $\textup{sgn}(v,e)=\textup{sgn}(v,e')$ if and only if $\textup{sgn}(f_V(v),f_E(e))=\textup{sgn}(f_V(v),f_E(e'))$.
\end{definition}

It turns out that each isomorphism class of local digraphs has a representation in graphs and a representation in digraphs, and both will be useful for us.
We first begin by defining the graphical representation.

\begin{center}
    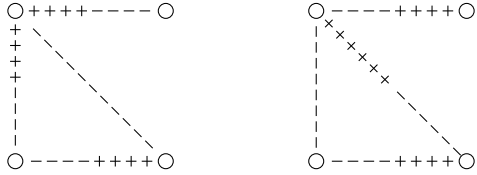
\begin{figure}[h]
        
        \begin{tikzpicture}[
  plusmarks/.style={
    postaction=decorate,
    decoration = {
            markings,
            mark = 
                between positions 0.08 and 0.92 step 6pt 
                with
                {
                    \draw (0pt, 2pt) -- (0pt, -2pt);
                    \draw (-2pt, 0pt) -- (2pt, 0pt);
                }
            }
  },
  minusmarks/.style={
    postaction=decorate,
    decoration = {
            markings,
            mark = 
                between positions 0.08 and 0.92 step 6pt 
                with
                {
                    \draw (-2pt, 0pt) -- (2pt, 0pt);
                }
            }
  },
  halfplus/.style={
    postaction=decorate,
    decoration={
            markings,
            mark = 
                between positions 0.08 and 0.45 step 6pt 
                with
                {
                    \draw (0pt, 2pt) -- (0pt, -2pt);
                    \draw (-2pt, 0pt) -- (2pt, 0pt);
                }
    }
  },
  halfminus/.style={
    postaction=decorate,
    decoration={markings,
      mark=between positions 0.55 and 0.92 step 6pt with
        {
                    \draw (-2pt, 0pt) -- (2pt, 0pt);}
    }
  }
        ]
        \begin{scope}
        \node[circle, draw, fill=white, inner sep=2pt] (A1) at (0,0) {};
        \node[circle, draw, fill=white, inner sep=2pt] (B1) at (2,0) {};
        \node[circle, draw, fill=white, inner sep=2pt] (C1) at (2,2) {};
        \node[circle, draw, fill=white, inner sep=2pt] (D1) at (0,2) {};
        \path[halfplus] (D1) -- (A1);
        \path[halfminus] (D1) -- (A1);
        \path[halfplus] (B1) -- (A1);
        \path[halfminus] (B1) -- (A1);
        \path[minusmarks] (B1) -- (D1);
        \path[halfplus] (D1) -- (C1);
        \path[halfminus] (D1) -- (C1);
        \end{scope}

        \begin{scope}[xshift=4cm]
        \node[circle, draw, fill=white, inner sep=2pt] (A2) at (0,0) {};
        \node[circle, draw, fill=white, inner sep=2pt] (B2) at (2,0) {};
        \node[circle, draw, fill=white, inner sep=2pt] (C2) at (2,2) {};
        \node[circle, draw, fill=white, inner sep=2pt] (D2) at (0,2) {};
        \path[minusmarks](A2)--(D2);
        \path[halfplus] (C2) -- (D2);
        \path[halfminus] (C2) -- (D2);
        \path[halfplus] (0,2) -- (2+0.1,0-0.1);
        \path[halfminus] (0,2) -- (2+0.05,0-0.05);
        \path[halfplus] (B2) -- (A2);
        \path[halfminus] (B2) -- (A2);
        \end{scope}
        \end{tikzpicture}

        \caption{Two isomorphic local digraphs}
        \label{fig:iso-digraph}
    \end{figure}
    
\end{center}
\begin{definition}[Graphification]\label{def:graphification}
    Let $G$ be a local digraph.
    Its \emph{graphification} $\gr(G)$ is a multigraph on $V(G)^+\sqcup V(G)^-$, where $V(G)^+$ and $V(G)^-$ are disjoint copies of $V(G)$ with signs, and 
    \[E(\gr(G)) = \left\{u^{\textup{sgn}(u,e)}v^{\textup{sgn}(v,e)}: e\in E(G)\textup{ with endpoints }u,v \right\}.\]
\end{definition}
See \cref{fig:graphification} for an example of graphification.

In the opposite direction, for any graph $G$ with a perfect matching $M$, we may reverse the process and define a local digraph with $M$ as its vertex set and $G\backslash M$ as its graphification.

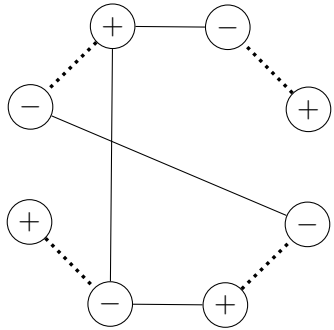
\begin{figure}[h]
    \begin{center}
        \begin{tikzpicture}
          \foreach \i/\lab in {22/+,67/-,112/+,157/-,202/+,247/-,292/+,337/-}
            \node[circle, draw, fill=white, inner sep=2pt] (v\i) at (\i:2cm) {$\lab$};
        
          \draw (v67)--(v112)--(v247)--(v292);
          \draw (v157)--(v337);
          \draw[dotted,line width=1.3pt] (v22)--(v67);
          \draw[dotted,line width=1.3pt] (v112)--(v157);
          \draw[dotted,line width=1.3pt] (v202)--(v247);
          \draw[dotted,line width=1.3pt] (v292)--(v337);
        \end{tikzpicture}
    \end{center}
    \caption{Graphification of the local digraph in \cref{fig:iso-digraph}. Dotted lines connect pairs of vertices that come from the same vertices in the original local digraph.}
    \label{fig:graphification}
\end{figure}

\begin{definition}\label{def:corresp-local-digraph}
    Let $G$ be a graph and $M$ be a perfect matching in $G$.
    Let $\ldg(G,M)$ be the local digraph whose vertex set is $\{v_e:e\in E(M)\}$ defined as follows.
    For every $e\in E(M)$, label its endpoints with a plus sign and a minus sign arbitrarily with a bijection $\sgn_e\colon e\to \{\pm\}$.
    Then $\ldg(G,M)$ is the unique local digraph whose graphification is $G\backslash M$.
    To be more specific, for each $e\in E(G)\backslash E(M)$, if $e_1,e_2\in E(M)$ are the two edges containing the endpoints of $e$, then create an edge $\Tilde{e}$ in $\ldg(G,M)$ between $e_1,e_2$ with $\sgn(v_{e_i},\Tilde{e}) = \sgn_{e_i}(e_i\cap e)$ for every $i\in\{1,2\}$.
\end{definition}

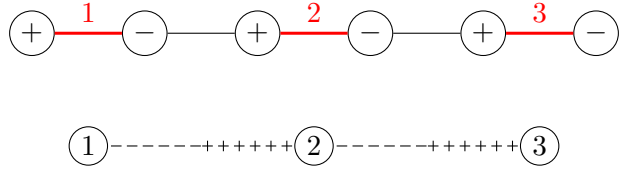
\begin{figure}[h]
    \begin{center}
        \begin{tikzpicture}[
  plusmarks/.style={
    postaction=decorate,
    decoration = {
            markings,
            mark = 
                between positions 0.08 and 0.92 step 6pt 
                with
                {
                    \draw (0pt, 2pt) -- (0pt, -2pt);
                    \draw (-2pt, 0pt) -- (2pt, 0pt);
                }
            }
  },
  minusmarks/.style={
    postaction=decorate,
    decoration = {
            markings,
            mark = 
                between positions 0.08 and 0.92 step 6pt 
                with
                {
                    \draw (-2pt, 0pt) -- (2pt, 0pt);
                }
            }
  },
  halfplus/.style={
    postaction=decorate,
    decoration={
            markings,
            mark = 
                between positions 0.04 and 0.47 step 6pt 
                with
                {
                    \draw (0pt, 2pt) -- (0pt, -2pt);
                    \draw (-2pt, 0pt) -- (2pt, 0pt);
                }
    }
  },
  halfminus/.style={
    postaction=decorate,
    decoration={markings,
      mark=between positions 0.53 and 0.96 step 6pt with
        {
                    \draw (-2pt, 0pt) -- (2pt, 0pt);}
    }
  }
        ]
            \begin{scope}
             \foreach \i/\lab in {1/+,2/-,3/+,4/-,5/+,6/-}
            \node[circle, draw, fill=white, inner sep=2pt] (v\i) at (1.5*\i,0) {$\lab$};
            \draw[line width=0.4mm, red ] (v1) -- (v2) node[midway,above] {$1$};
            \draw (v2)--(v3);
            \draw[line width=0.4mm, red] (v3)--(v4) node[midway,above]{$2$};
            \draw (v4)--(v5);
            \draw[line width=0.4mm, red] (v5)--(v6) node[midway,above]{$3$};
            \end{scope}

            \begin{scope}[yshift=-1.5cm]
                \foreach \i in {1,2,3}
            \node[circle, draw, fill=white, inner sep=2pt] (v\i) at (3*\i-0.75,0) {$\i$};
            \path[halfplus] (v2)--(v1);
            \path[halfminus] (v2)--(v1);
            \path[halfplus] (v3)--(v2);
            \path[halfminus] (v3)--(v2);
            \end{scope}
          
        \end{tikzpicture}
    \end{center}
    \caption{$P_6$ with the perfect matching $M$, and the corresponding local digraph $\ldg(P_6,M)$.}
    \label{fig:ldg}
\end{figure}

To end this subsection, we define an operation that lifts a local digraph to a digraph.

\begin{definition}[Double cover]\label{def:double-cover}
    Let $G$ be a local digraph.
    Its \emph{double cover} $\overline{G}$ is a multi-digraph on $V(G)^+\sqcup V(G)^-$ so that for every edge $e\in E(G)$ with $e=\{uv\}$, we have two corresponding directed edges: $e_1$ that goes from $u^{\textup{sgn}(u,e)}$ to $v^{-\textup{sgn}(v,e)}$, and $e_2$ that goes from $v^{\textup{sgn}(v,e)}$ to $u^{-\textup{sgn}(u,e)}$.
\end{definition}

Note that each vertex is replaced with two vertices, and each edge corresponds to two directed edges in the digraph, which is why we call it a double cover.
We remark that this double cover is ``trivial'', i.e.\ it splits into two layers of digraphs, if and only if the local digraph corresponds to some digraph.
As we will not use this, we omit the proof here.

\begin{figure}[h]
    \centering
        \begin{tikzpicture}[
  plusmarks/.style={
    postaction=decorate,
    decoration = {
            markings,
            mark = 
                between positions 0.08 and 0.92 step 6pt 
                with
                {
                    \draw (0pt, 2pt) -- (0pt, -2pt);
                    \draw (-2pt, 0pt) -- (2pt, 0pt);
                }
            }
  },
  minusmarks/.style={
    postaction=decorate,
    decoration = {
            markings,
            mark = 
                between positions 0.08 and 0.92 step 6pt 
                with
                {
                    \draw (-2pt, 0pt) -- (2pt, 0pt);
                }
            }
  },
  halfplus/.style={
    postaction=decorate,
    decoration={
            markings,
            mark = 
                between positions 0.08 and 0.45 step 6pt
                with
                {
                    \draw (0pt, 2pt) -- (0pt, -2pt);
                    \draw (-2pt, 0pt) -- (2pt, 0pt);
                }
    }
  },
  halfminus/.style={
    postaction=decorate,
    decoration={markings,
      mark=between positions 0.55 and 0.92 step 6pt with
        {
                    \draw (-2pt, 0pt) -- (2pt, 0pt);}
    }
  }
        ]
        \begin{scope}
        \node[circle, draw, fill=white, inner sep=2pt] (A1) at (0,0) {};
        \node[circle, draw, fill=white, inner sep=2pt] (B1) at (2,0) {};
        \node[circle, draw, fill=white, inner sep=2pt] (C1) at (2,2) {};
        \node[circle, draw, fill=white, inner sep=2pt] (D1) at (0,2) {};
        \path[halfplus] (D1) -- (A1);
        \path[halfminus] (D1) -- (A1);
        \path[halfplus] (B1) -- (A1);
        \path[halfminus] (B1) -- (A1);
        \path[minusmarks] (B1) -- (D1);
        \path[halfplus] (D1) -- (C1);
        \path[halfminus] (D1) -- (C1);
        \end{scope}

        \begin{scope}[xshift=4cm,yshift=-1cm]
            \foreach \x/\y in {0/0,2/0,0/2,2/2} \node[circle,draw,fill=white,inner sep=2pt] (plus\x\y) at (\x*1.5,\y*1.5) {$+$};
            \foreach \x/\y in {0/0,2/0,0/2,2/2} \node[circle,draw,fill=white,inner sep=2pt] (minus\x\y) at (\x*1.5+1,\y*1.5+1) {$-$};
            \tikzset{>={Stealth[scale=1.6]}};
            \draw[>->] (plus20)--(plus00);
            \draw[>->] (minus00)--(minus20);
            \draw[>->] (plus02)--(plus00);
            \draw[>->] (minus00)--(minus02);
            \draw[>->] (plus02)--(plus22);
            \draw[>->] (minus22)--(minus02);
            \draw[>->] (minus20)--(plus02);
            \draw[>->] (minus02)--(plus20);
            
          \draw[dotted,line width=1.3pt] (plus20)--(minus20);
          \draw[dotted,line width=1.3pt] (plus02)--(minus02);
          \draw[dotted,line width=1.3pt] (plus00)--(minus00);
          \draw[dotted,line width=1.3pt] (plus22)--(minus22);
        \end{scope}

    \end{tikzpicture}
    \caption{A local digraph and its double cover. Dotted lines connect vertices in the double cover that correspond to the same vertex of the local digraph.}
    \label{fig:double-cover}
\end{figure}
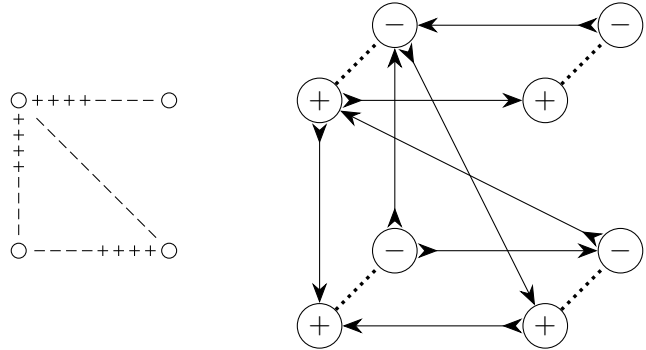

\subsection{Locally directed acyclic graphs}

Having defined what local digraphs are, we now study the local-digraph analog of acyclicity.
Later, we will show that this is the property that a local digraph would naturally possess if it comes from a graph with a perfect matching being its unique fractional perfect matching.
Along the way, we will also prove some useful properties of acyclic local digraphs.

\begin{definition}[Locally directed acyclic graph]\label{def:ldag}
    Given a local digraph $G$. A \emph{locally directed closed walk} in $G$ is a walk $v_1\xrightarrow{e_1}v_2\xrightarrow{e_2}\cdots\xrightarrow{e_{t-1}}v_t\xrightarrow{e_t=e_0}v_{t+1}=v_1$ with $\textup{sgn}(v_i,e_i)\neq \textup{sgn}(v_i,e_{i-1})$ for all $i\in[t]$.
    We call $G$ a \emph{locally directed acyclic graph} or \emph{LDAG} if it does not have any locally directed closed walk.
\end{definition}

We would like to warn the readers that unlike the case for graphs or digraphs, having a locally directed closed walk is different from having a locally directed closed cycle, a closed walk that uses no repeated vertices.
For example, the local digraph in \cref{fig:no-directed-cycle} has a directed closed walk $514523$, yet the only cycles in the local digraph are $514$ and $523$, both of which are not locally directed at $5$.
However, it is true that each locally directed closed walk can be reduced to a locally directed closed walk that uses each vertex at most twice.
We will also omit the proof of this claim as this will not be used, though it can be extracted from the proof of \cref{prop:ldag-double-cover}.

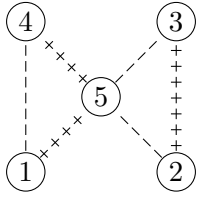
\begin{figure}[h]
    \centering
        \begin{tikzpicture}[
  plusmarks/.style={
    postaction=decorate,
    decoration = {
            markings,
            mark = 
                between positions 0.06 and 0.94 step 6pt 
                with
                {
                    \draw (0pt, 2pt) -- (0pt, -2pt);
                    \draw (-2pt, 0pt) -- (2pt, 0pt);
                }
            }
  },
  minusmarks/.style={
    postaction=decorate,
    decoration = {
            markings,
            mark = 
                between positions 0.06 and 0.94 step 6pt 
                with
                {
                    \draw (-2pt, 0pt) -- (2pt, 0pt);
                }
            }
  },
  diaplus/.style={
    postaction=decorate,
    decoration = {
            markings,
            mark = 
                between positions 0.09 and 0.91 step 6pt 
                with
                {
                    \draw (0pt, 2pt) -- (0pt, -2pt);
                    \draw (-2pt, 0pt) -- (2pt, 0pt);
                }
            }
  },
  diaminus/.style={
    postaction=decorate,
    decoration = {
            markings,
            mark = 
                between positions 0.09 and 0.91 step 6pt 
                with
                {
                    \draw (-2pt, 0pt) -- (2pt, 0pt);
                }
            }
  },]
        \node[circle, draw, fill=white, inner sep=2pt] (A1) at (0,0) {$1$};
        \node[circle, draw, fill=white, inner sep=2pt] (B1) at (2,0) {$2$};
        \node[circle, draw, fill=white, inner sep=2pt] (C1) at (2,2) {$3$};
        \node[circle, draw, fill=white, inner sep=2pt] (D1) at (0,2) {$4$};
        \node[circle, draw, fill=white, inner sep=2pt] (E1) at (1,1) {$5$};
        \path[minusmarks] (A1)--(D1);
        \path[diaminus] (B1)--(E1);
        \path[diaminus] (E1)--(C1);
        \path[plusmarks] (B1)--(C1);
        \path[diaplus] (A1)--(E1);
        \path[diaplus] (E1)--(D1);

    \end{tikzpicture}
    \caption{A local digraph with a directed closed walk but no directed cycle}
    \label{fig:no-directed-cycle}
\end{figure}

Next, we show that acyclicity is preserved under taking double covers.

\begin{proposition}\label{prop:ldag-double-cover}
    A local digraph is acyclic if and only if its double cover is acyclic.
\end{proposition}
\begin{proof}
    We first show that if a local digraph has a locally directly closed walk, then its double cover has a directed closed walk.
    To show this, it suffices to consider the double cover of a locally directed closed walk $v_1\xrightarrow{e_1}v_2\xrightarrow{e_2}\cdots\xrightarrow{e_{t-1}}v_t\xrightarrow{e_t=e_0}v_{t+1}=v_1$.
    Now set $s_i =  \sgn(v_i,e_i)$ for each $i\in[t]$ and set $s_{t+1}=s_1$.
    Then one of the edges that $e_i$ is lifted up to is $v_i^{s_i}\to v_{i+1}^{-\sgn(v_{i+1},e_i)}$, which is simply $v_i^{s_i}\to v_{i+1}^{s_{i+1}}$ as $s_{i+1}=\sgn(v_{i+1},e_{i+1})\neq \sgn(v_{i+1},e_i)$.
    Therefore we have found a directed closed walk $v_1^{s_1}\to v_2^{s_2}\to\cdots\to v_t^{s_t}\to v_1^{s_1}$, as desired.
    Reversing each step, we can also see that each directed closed walk in the double cover projects down to a locally directed closed walk in the original local digraph.
\end{proof}

\begin{figure}[h]
    \centering
        \begin{tikzpicture}[
  plusmarks/.style={
    postaction=decorate,
    decoration = {
            markings,
            mark = 
                between positions 0.08 and 0.92 step 6pt 
                with
                {
                    \draw (0pt, 2pt) -- (0pt, -2pt);
                    \draw (-2pt, 0pt) -- (2pt, 0pt);
                }
            }
  },
  minusmarks/.style={
    postaction=decorate,
    decoration = {
            markings,
            mark = 
                between positions 0.08 and 0.92 step 6pt 
                with
                {
                    \draw (-2pt, 0pt) -- (2pt, 0pt);
                }
            }
  },
  halfplus/.style={
    postaction=decorate,
    decoration={
            markings,
            mark = 
                between positions 0.08 and 0.45 step 6pt 
                with
                {
                    \draw (0pt, 2pt) -- (0pt, -2pt);
                    \draw (-2pt, 0pt) -- (2pt, 0pt);
                }
    }
  },
  halfminus/.style={
    postaction=decorate,
    decoration={markings,
      mark=between positions 0.55 and 0.92 step 6pt with
        {
                    \draw (-2pt, 0pt) -- (2pt, 0pt);}
    }
  }
        ]
        \begin{scope}
        \node[circle, draw, fill=white, inner sep=2pt] (A1) at (0,0) {};
        \node[circle, draw, fill=white, inner sep=2pt] (B1) at (2,0) {};
        \node[circle, draw, fill=white, inner sep=2pt] (C1) at (2,2) {};
        \node[circle, draw, fill=white, inner sep=2pt] (D1) at (0,2) {};
        \path[halfplus] (D1) -- (A1);
        \path[halfminus] (D1) -- (A1);
        \path[plusmarks] (A1) -- (B1);
        \path[minusmarks] (B1) -- (D1);
        \path[halfplus] (D1) -- (C1);
        \path[halfminus] (D1) -- (C1);
        \end{scope}

        \begin{scope}[xshift=4cm,yshift=-1cm]
            \foreach \x/\y in {0/0,2/0,0/2,2/2} \node[circle,draw,fill=white,inner sep=2pt] (plus\x\y) at (\x*1.5,\y*1.5) {$+$};
            \foreach \x/\y in {0/0,2/0,0/2,2/2} \node[circle,draw,fill=white,inner sep=2pt] (minus\x\y) at (\x*1.5+1,\y*1.5+1) {$-$};
            \tikzset{>={Stealth[scale=1.6]}};
            \draw[>->,red] (plus20)--(minus00);
            \draw[>->,blue] (plus00)--(minus20);
            \draw[>->,blue] (plus02)--(plus00);
            \draw[>->,red] (minus00)--(minus02);
            \draw[>->] (plus02)--(plus22);
            \draw[>->] (minus22)--(minus02);
            \draw[>->,blue] (minus20)--(plus02);
            \draw[>->,red] (minus02)--(plus20);
          \draw[dotted,line width=1.3pt] (plus20)--(minus20);
          \draw[dotted,line width=1.3pt] (plus02)--(minus02);
          \draw[dotted,line width=1.3pt] (plus00)--(minus00);
          \draw[dotted,line width=1.3pt] (plus22)--(minus22);
        \end{scope}

    \end{tikzpicture}
    \caption{A local digraph with a locally directed closed walk, and its double cover with directed cycles. Dotted lines connect vertices in the double cover that correspond to the same vertex of the local digraph.}
    \label{fig:cyclic-double-cover}
\end{figure}
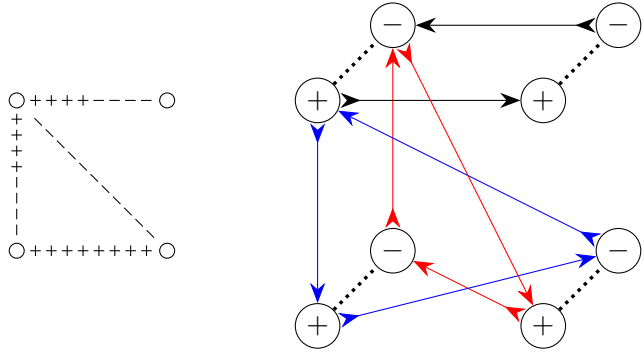

For example, in \cref{fig:double-cover} we see that both the local digraph and its double cover are acyclic, whereas in \cref{fig:cyclic-double-cover} we can find (locally) directed closed walks in both the local digraph and its double cover.

Now we can transfer properties of acyclic digraphs, also called DAGs, to properties of LDAGs. 
One property that we need regards what is known as \emph{topological sorting}.

\begin{definition}[Topological sort]\label{def:topological-sort}
    For any LDAG $G$, an ordering $v_1,\ldots, v_n$ of $V(G)$ is a \emph{topological sort} of $G$ if there is a choice of signs $s_1,\ldots,s_n$ such that for any $1\leq i<j\leq n$ and any edge $e$ between $v_i$ and $v_j$, we have  $\textup{sgn}(v_i,e)=s_i$.
\end{definition}

\begin{proposition}\label{prop:topological-sort}
    Every LDAG admits a topological sort.
\end{proposition}
\begin{proof}
    Let $G$ be an LDAG with $n$ vertices, and let $u_1,\ldots, u_{2n}$ be a topological sort of the double cover $\overline{G}$ of $G$, so that each directed edge goes from $u_i$ to $u_j$ for some $j>i$. 
    For each $v\in V(G)$, delete the later appearance of $v^+$ and $v^-$ in the ordering $u_1,\ldots,u_{2n}$. 
    Suppose that $v_1^{s_1},\ldots, v_n^{s_n}$ is the resulting ordering.
    We claim that any edge $e\in E(G)$ between $v_i,v_j$ with $i<j$ must satisfy $\textup{sgn}(v_i,e)=s_i$, which would imply that $v_1,\ldots, v_n$ is a topological sort for $G$.

    For the sake of contradiction, assume that $\sgn(v_i,e)=-s_i$ for some $1\leq i<j\leq n$ and edge $e$ between $v_i$ and $v_j$.
    Then we know that there is a directed edge in $\overline{G}$ going from $v_j^{\sgn(v_j,e)}$ to $v_i^{-\sgn(v_i,e)} = v_i^{s_i}$.
    However, by the choice of the ordering and the signs, we know that $v_j^{\sgn(v_j,e)}$ does not appear earlier than $v_j^{s_j}$, which appears later than $v_i^{s_i}$ in the topological sort of $\overline{G}$.
    This is a contradiction, as desired.
\end{proof}

Now we can prove the main proposition of this subsection, which explains why we are interested in LDAG.

\begin{proposition}\label{prop:unique-pm-imply-ldag}
    Let $G$ be a graph and $M$ be a perfect matching.
    Then $M$ is the only fractional perfect matching of $G$ if and only if $\ldg(G,M)$ is acyclic.
\end{proposition}
\begin{proof}
    We first show that if there is another fractional perfect matching, then $\ldg(G,M)$ is not acyclic.
    Let $f\colon E(G)\to [0,1]$ be such that $\sum_{e\ni v}f(e)=1$ for every $v\in V(G)$.
    Now whenever $f(e)=1$ for any edge $e$ in the matching $M$, remove both its endpoints together with all edges that use any of the two endpoints from $G$, and also remove it from $M$.
    We repeat this process until we end up with a graph $G'$ and a perfect matching $M'$ where $f(e)<1$ for every $e\in M'$.
    Then we see that $\ldg(G',M')$ is an induced subgraph of $\ldg(G,M)$ as local digraphs, and so it suffices to show that $\ldg(G',M')$ is not acyclic.
    Note that since $f(e)<1$ for every $e\in M'$, any $v\in V(G')$ is incident to some edge not in $M'$.
    As a consequence, for any vertex $v\in V(\ldg(G',M'))$ and sign $s\in\{\pm\}$ there exists an edge $e\in E(\ldg(G',M'))$ with $\sgn(v,e)=s$.
    However, this shows that there cannot be any topological sort for $\ldg(G',M')$ as no vertex can be the first in the ordering, which shows that $\ldg(G',M')$ is not acyclic by \cref{prop:topological-sort}.

    Now suppose that $\ldg(G,M)$ is not acyclic, which shows that there is a locally directed closed walk in $\ldg(G,M)$.
    Suppose that $v_1\xrightarrow{e_1}v_2\xrightarrow{e_2}\cdots\xrightarrow{e_{t-1}}v_t\xrightarrow{e_t=e_0}v_{t+1}=v_1$ is the locally directed closed walk, where each $v_i\in V(\ldg(G,M))$ corresponds to the edge $v_i^+v_i^-$ in the matching $M$.
    Let $u_i=v_i^{\sgn(v_i,e_{i-1})}$ and $w_i=v_{i}^{\sgn(v_i,e_i)}$ for each $i\in[t]$.
    By the definition of $\ldg(G,M)$, we know that the edges $e_1,\ldots, e_t$ in $\ldg(G,M)$ corresponds to the edges $w_1u_2,w_2u_3,\ldots, w_tu_1$ in $G$.
    By the definition of locally directed closed walk, we see that $u_i$ and $w_i$ must be two different endpoints of some edge in $M$, showing that $u_iw_i\in M$.
    Note that this shows that $u_1w_1\cdots u_tw_tu_1$ is an alternating closed walk for the matching $M$, which means that we can swap some weight on $M$ using the alternating walk to get another fractional perfect matching.
    To be more specific, let $c\colon E(G)\to\ZZ_{\geq 0}$ be the function that counts, for each edge $e$, how many times it appears as an edge in the walk $u_1w_1\cdots u_tw_tu_{1}$, and for every $\varepsilon>0$ let $f_{\varepsilon}:E(G)\to \RR$ be such that $f_{\varepsilon}(e)=\varepsilon c(e)$ if $e\not\in M$, and $f_{\varepsilon}(e) = 1-\varepsilon c(e)$ if $e\in M$.
    It is easy to verify that $f_{\varepsilon}$ is another fractional perfect matching for all $\varepsilon \leq \left(\max_{e\in E(G)}c(e)\right)^{-1}$.
\end{proof}

We end with an auxiliary lemma that will be useful later when we would like to show that iterated blowup preserves acyclicity.

\begin{lemma}\label{lem:LDAG-blowup}
    Let $G_1,G_2$ be two LDAGs, and let $v$ be a vertex of $G_1$.
    Let $G$ be a graph where, starting from $G_1$, $v$ is blown-up to a set of size $v(G_2)$ in $G_1$, and $G_2$ is put on this set.
    Then $G$ is an LDAG as well.
\end{lemma}
\begin{proof}
    Suppose that the vertex set $V(G)$ of $G$ can be written as $(V(G_1)\backslash \{v\})\cup \{v^u:u\in V(G_2)\}$, where $G$ induced on $\{v^u:u\in V(G_2)\}$ is isomorphic to $G_2$ via the bijection $v^u\mapsto u$ for each $u\in V(G_2)$.
    Let $v_1,\ldots, v_s$ be a topological sort of $G_1$ (where $s=v(G_1)$) and $u_1,\ldots, u_t$ be a topological sort of $G_2$ (where $t=v(G_2)$), which exist by \cref{prop:topological-sort}.
    Suppose that $v_i=v$.
    Then it is clear that $v_1,\ldots, v_{i-1}, v^{u_1},\ldots, v^{u_t},v_{i+1},\ldots, v_s$ is a topological sort of $G$, showing that $G$ is an LDAG once again by \cref{prop:topological-sort}.
\end{proof}

\subsection{From edge inducibility to vertex inducibility}

In this section, we will generalize \cref{thm:eind-as-hard-as-ind} to any graph with a perfect matching that has no other fractional perfect matchings.
It turns out that the edge inducibilities of such graphs are related to vertex inducibilities of certain local digraphs.
We will thus have to define vertex inducibility for local digraphs.

\begin{definition}[Vertex inducibility and acyclic vertex inducibility for local digraphs]\label{def:local-digraph-inducibility}
    Let $G$ be any local digraph.
    For any local digraph $H$, let $N_{\ind}(G,H)$ be the number of induced local digraphs of $H$ that are isomorphic to $G$.
    The \emph{vertex inducibility} of $G$ is
    \[\ind(G)\eqdef\limsup_{n\to\infty}\max_{v(H)=n}\frac{N_{\ind}(G,H)}{\binom{n}{v(G)}}\]
    where the maximum ranges through all local digraphs of size $n$.
    Its \emph{acyclic vertex inducibility} is denoted by $\aind(G)$, and is defined by the same expression except that $H$ is required to be acyclic.
    In particular, if $G$ is not an LDAG, then $\aind(G)=0$.
\end{definition}

\begin{theorem}\label{thm:aind-eind-ind}
    Let $G$ be a graph and $M$ be a perfect matching in $G$ so that $M$ is a unique fractional perfect matching.
    Then 
    \[\aind(\ldg(G,M))\leq \frac{2^{\abs{M}}\abs{M}!}{\abs{\Aut(G)}}\eind(G)\leq \ind(\ldg(G,M)).\]
\end{theorem}
\begin{proof}
    We first prove the second half of the inequality.
    The proof is similar to that of \cref{thm:eind-as-hard-as-ind}.
    Suppose that $H$ is a graph with $m$ edges.
    We will construct a local digraph $H'$ with $m$ vertices such that every induced copy of $G$ in $H$ gives rise to many distinct induced copies of $\ldg(G,M)$ in $H'$. This is done as follows.
    For each $e\in E(H)$, let $v_e$ be a vertex in $H'$ and if $e=\{u,v\}$, fix any bijection $\sgn_e\colon\{u,v\}\to \{\pm\}$.
    For any edges $e,e_1,e_2\in E(H)$ where $e_1,e_2$ are disjoint and $e$ goes between $e_1$ and $e_2$, we create an edge $\Tilde{e}$ between $v_{e_1},v_{e_2}$. 
    We also set
    \[\sgn(v_{e_i},\Tilde{e}) = \sgn_{e_i}(e_i\cap e)\]
    for $i\in\{1,2\}$.
     From this construction and the definition of $\ldg(G,M)$, in any induced copy of $G$ in $H$, the vertices corresponding to the unique perfect matching of the induced copy induce a local digraph isomorphic to $\ldg(G,M)$ in $H'$.
    Therefore, we know that 
    \[\frac{\abs{M}!\cdot N_{\ind}(\ldg(G,M),H')}{m^{\abs{M}}}\geq \frac{\abs{M}!\cdot N_{\ind}(G,H)}{m^{\abs{M}}}= \frac{2^{\abs{M}}\abs{M}!}{\abs{\Aut(G)}}\cdot\frac{ \abs{\Aut(G)}N_{\ind}(G,H)}{(2m)^{\abs{M}}}.\]
    By the definitions of edge inducibility for graphs and vertex inducibility for local digraphs, we get that
    \[\ind(\ldg(G,M))\geq  \frac{2^{\abs{M}}\abs{M}!}{\abs{\Aut(G)}}\eind(G),\]
    as desired.

    To show the first half of the inequality, we begin with an LDAG $H$ with $n$ vertices.
    We will construct a graph $H'$ with many induced copies of $G$ corresponding to indueced copies of $\ldg(G,M)$ in $H$, though the number of edges of $H'$ would not be $n$.
    Instead, we will construct a graph $H'$ with $(n+o_{n;k\to\infty}(1))k$ edges where $k\in\NN$ is a sufficiently large positive integer.
    To do so, let $v_1,\ldots, v_n$ be a topological sort of $H$, and also relabel the signs so that $\sgn(v_i,e)=+$ for any $e$ going between $v_i,v_j$ and $i<j$.
    Take the graphification $\gr(H)$ of $H$, and let $M_H$ be the perfect matching $\{v_1^+v_1^-,\ldots, v_n^+v_n^-\}$.
    Then any induced copy of $\ldg(G,M)$ in $H$ corresponds to an induced copy of $G$ in $\gr(H)\cup M_H$ where the unique matching $M$ of $G$ is mapped to a submatching of $M_H$.
    We can also, without losing any induced copies of $G$, assume that $\gr(H)\cup M_H$ is simple.
    The graph $H'$ will then be chosen to be an appropriate blowup of $\gr(H)\cup M_H$ constructed as follows.
    Fix some rational numbers $0<\alpha_1<\cdots <\alpha_n<1/2$, and let $k$ be a large positive integer such that $k^{\alpha_i}$ is also a positive integer for any $i\in[n]$.
    Let $H'$ be the blowup of $\gr(H)\cup M_H$ where $v_i^+$ is replaced by an independent set of size $k^{\alpha_i}$ and $v_i^-$ is replaced by an independent set of size $k^{1-\alpha_i}$.
    Then the number of edges in $H'$ is at least $nk$ and can be upper bounded by
    \[nk+\sum_{1\leq i<j\leq n}k^{\alpha_i}\left(k^{\alpha_j}+k^{1-\alpha_j}\right) = (n+o_{n;k\to\infty}(1))k.\]
    This shows that $e(H') = (n+o_{n;k\to\infty}(1))k$.
    Moreover, for any induced copy of $G$ in $\gr(H)\cup M_H$ whose unique perfect matching $M$ is a subset of $M_H$, it is blown-up to $k^{\abs{M}}$ induced copies in $H'$.
    Therefore,
    \[ N_{\ind}(\ldg(G,M),H)\leq  \frac{N_{\ind}(G,H')}{k^{\abs{M}}}
        =(1+o_{n,\abs{M};k\to\infty}(1))(2n)^{\abs{M}}\cdot \frac{N_{\ind}(G,H')}{(2e(H'))^{\abs{M}}}.\]
    By taking $k$ to infinity, we see that
    \[N_{\ind}(\ldg(G,M),H)\leq \frac{(2n)^{\abs{M}}}{\abs{\Aut(G)}}\eind(G),\]
    and so
     \[\frac{\abs{M}!\cdot N_{\ind}(\ldg(G,M),H)}{n^{\abs{M}}}\leq \frac{2^{\abs{M}}\abs{M}!}{\abs{\Aut(G)}}\eind(G).\]
    By the definition of acyclic vertex inducibility, we see that $\aind(\ldg(G,M))\leq \frac{2^{\abs{M}}\abs{M}!}{\abs{\Aut(G)}}\eind(G)$, as desired.
\end{proof}

We remark that now \cref{thm:eind-as-hard-as-ind} becomes a corollary of this theorem.
Indeed, if $G$ is a graph and $G'$ is obtained by adding a pendant edge to each vertex of $G$, then the pendant edges form a perfect matching $M$ for $G'$ that is the unique fractional perfect matching.
In addition, $\ldg(G',M)$ is a local digraph that is $G$ as a graph, and all signs of the edges at any endpoint are the same.
In this case, it is easy to verify that
\[\aind(\ldg(G',M))=\ind(\ldg(G',M)) = \ind(G).\]

In general, in the setting of \cref{thm:aind-eind-ind}, if we know that $\aind(\ldg(G,M)) = \ind(\ldg(G,M))$, then we immediately know the value of $\eind(G)$.
This will be used later to determine edge inducibilities of some small graphs.
For now, we first apply \cref{thm:aind-eind-ind} to prove \cref{thm:better-construction-intro}, which is an improved lower bound compared to the simple lower bound given by \cref{cor:simple bound}.

\begin{proof}[Proof of \cref{thm:better-construction-intro}]
    We first note that $\aind(G') \geq \frac{k!}{k^k-k}$ for any LDAG $G'$ on $k$ vertices.
    Indeed, we can simply take the iterated blowups of $G'$, which are still acyclic by \cref{lem:LDAG-blowup}.
    The iterated blowups then show that $\aind(G')\geq \frac{k!}{k^k-k}$, as desired.

    Now by \cref{prop:unique-pm-imply-ldag}, we know that $\ldg(G,M)$ is an LDAG on $\abs{M}=k$ vertices.
    By \cref{thm:aind-eind-ind}, we thus have
    \[\eind(G)\geq \frac{\abs{\Aut(G)}}{2^{k}\cdot k!}\aind(\ldg(G,M))\geq \frac{\abs{\Aut(G)}}{2^{k}\cdot k!}\cdot \frac{k!}{k^{k}-k} = \frac{\abs{\Aut(G)}}{2^{k}\left(k^{k}-k\right)},\]
    which is the desired inequality.
\end{proof}

\begin{remark}
    In the recent independent work of Wang, Zhao, and Lu \cite{WanZhaLu2025}, they prove that 
    \[\frac{1}{2^{t-1}(2t+1)^{t-1}}\leq \eind(P_{2t})\leq \frac{1}{2^t(t-1)^{t-1}}\]
    when $t\geq 2$, and conjecture that the lower bound is tight.
    However, by \cref{thm:better-construction-intro} and the fact that $P_{2t}$ has a perfect matching $M$ that is the unique fractional perfect matching, we see that in fact 
    \[\eind(P_{2t})\geq \frac{1}{2^{t-1}(t^t-t)},\]
    disproving their conjecture for all $t\geq 2$.
    This lower bound also determines $\eind(P_{2t})$ within a multiplicative factor of $O(t)$.
    We also remark that, as observed by Even-Zohar and Linial \cite{EveLin2015}, an appropriate modification of Pippenger and Golumbic's upper bound for inducibilities of cycles \cite[Theorem 9]{PipGol1975} yields $\ind(P_t)\leq \frac{t!}{2(t-1)^{t-1}}$.
    Since $\ldg(P_{2t},M)$ is the locally directed path $LDP_t$ on $t$ vertices, \cref{thm:aind-eind-ind} gives
    \[\eind(P_{2t})\leq \frac{2}{2^t\cdot t!}\ind(LDP_t)\leq \frac{2}{2^t\cdot t!}\ind(P_t)\leq \frac{1}{2^t(t-1)^{t-1}},\]
    recovering Wang, Zhao, and Lu's upper bound.
    We suspect that $\ind(LDP_t)$ is always strictly smaller than $\ind(P_t)$ when $t\geq 3$, which would suggest that the upper bound is not tight either.
\end{remark}

We end with the particular case $G=P_6$, where we have an even better construction also coming from an application of \cref{thm:aind-eind-ind}.
We will also prove an upper bound on $\eind(P_6)$ that is better than what we obtained in the previous remark.

\begin{theorem}\label{thm:P6}
    We have
    \[\frac{5}{372}\leq \eind(P_6) \leq \frac{1}{36}.\]
\end{theorem}
\begin{proof}
    Let $M$ be the unique matching of $P_6$ with three edges.
    It is clear that $P_6$ has no other fractional perfect matchings, and so we can apply \cref{thm:aind-eind-ind}.
    Note that $\abs{M} = 3$ and $\abs{\Aut(P_6)}=2$.
    We also know what $\ldg(P_6,M)$ is from \cref{fig:ldg}: it is the locally directed path on three vertices, which we denote as $LDP_3$.
    Then by \cref{thm:aind-eind-ind}, it suffices to show that $\ind(LDP_3)\leq\frac{2}{3}$ and $\aind(LDP_3)\geq \frac{10}{31}$.

    We first show that $\ind(LDP_3)\leq \frac{2}{3}$ using a simple double counting argument.
    For any local digraph $H$ on $n$ vertices, we will show that there are at most $\frac{n^3}{9}$ induced copies of $LDP_3$.
    We first note that we may assume $H$ is simple.
    Now let $A$ be the number of triplets $\{u,v,w\}$ with $uv,uw\in E(G)$, $\sgn(u,uv)=+$ and $\sgn(u,uw)=-$.
    Let $B$ be the number of triplets $\{u,v,w\}$ with $uv\in E(G), uw\not\in E(G)$.
    It is clear $N_{\ind}(LDP_3,H)\leq \min\{A,\frac{1}{2}B\}\leq \frac{1}{3}A+\frac{1}{3}B$.
    Hence,
    \begin{align*}
        N_{\ind}(LDP_3,H)\leq &\sum_{u\in V(H)}\left(\frac{1}{3}\left(\frac{d_H(u)}{2}\right)^2+\frac{1}{3}d_H(u)(n-d_H(u))\right)\\ = &\sum_{u\in V(H)}\left(\frac{1}{3}d_H(u)n-\frac{1}{4}d_H(u)^2\right)\\ =&\sum_{u\in V(H)}\left(\frac{1}{9}n^2-\left(\frac{1}{3}n-\frac{1}{2}d_H(u)\right)^2\right)\\
        &\leq \frac{1}{9}n^3,
    \end{align*}
    as desired.

    Now, to show that $\aind(LDP_3)\geq \frac{10}{31}$, we will begin by constructing, for each $k$, an LDAG $G_k$ on $4k$ vertices so that $N_{\ind}(LDP_3,G_k)= (\frac{5}{16}+o(1))\binom{4k}{3}$.
    Later, we will iteratively blowup this construction to get the desired construction.
    
    The construction of $G_k$ goes as follows.
    First, consider $2k+2$ vertices $v_1,v_2,\ldots, v_{2k},u_1,u_2$.
    Connect $v_iv_j$ if $2\nmid i-j$, connect $u_1u_2$ and connect $v_iu_j$ if $2\nmid i-j$ as well.
    We will orient the edges locally so that $v_1,\ldots, v_{2k},u_1,u_2$ is a valid topological sort.
    To achieve this, set $\textup{sgn}(v_i,v_iv_j)=+$ if $i<j$ and $2\nmid i-j$ and $\textup{sgn}(v_i,v_iu_j)=+$ if $2\nmid i-j$.
    For $i>j$ with $2\nmid i-j$, we set $\textup{sgn}(v_i,v_iv_j)=-$.
    Finally, we set $\textup{sgn}(u_j,v_iu_j)=-$ for every $2\nmid i-j$ and $\textup{sgn}(u_i,u_iu_j)=+$ for every permutation $(i,j)$ of $(1,2)$.
    This clearly makes the graph an LDAG as $v_1,\ldots, v_{2k},u_1,u_2$ is a topological sort.
    The LDAG $G_k$ is where we blowup $u_1$ and $u_2$ with independent sets $\{u_1^{(1)},\ldots, u_1^{(k)}\},\{u_2^{(1)},\ldots, u_2^{(k)}\}$ of size $k$.

    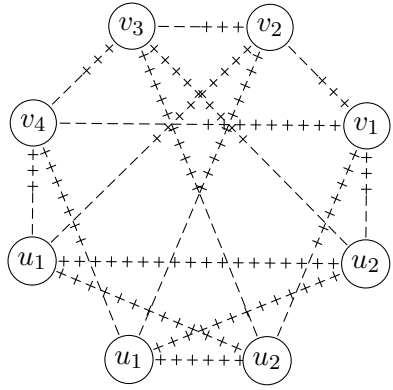
\begin{figure}[h]
    \begin{center}
        \begin{tikzpicture}[
  plusmarks/.style={
    postaction=decorate,
    decoration = {
            markings,
            mark = 
                between positions 0 and 1 step 6pt 
                with
                {
                    \draw (0pt, 2pt) -- (0pt, -2pt);
                    \draw (-2pt, 0pt) -- (2pt, 0pt);
                }
            }
  },
  minusmarks/.style={
    postaction=decorate,
    decoration = {
            markings,
            mark = 
                between positions 0 and 0.97 step 6pt 
                with
                {
                    \draw (-2pt, 0pt) -- (2pt, 0pt);
                }
            }
  },
  halfplus/.style={
    postaction=decorate,
    decoration={
            markings,
            mark = 
                between positions 0.08 and 0.45 step 6pt 
                with
                {
                    \draw (0pt, 2pt) -- (0pt, -2pt);
                    \draw (-2pt, 0pt) -- (2pt, 0pt);
                }
    }
  },
  halfminus/.style={
    postaction=decorate,
    decoration={markings,
      mark=between positions 0.55 and 0.92 step 6pt with
        {
                    \draw (-2pt, 0pt) -- (2pt, 0pt);}
    }
  }
        ]
          
          \foreach \i/\lab/\nd in {22/v_1/v1,67/v_2/v2,112/v_3/v3,157/v_4/v4,202/u_1/u11,247/u_1/u12,292/u_2/u21,337/u_2/u22}
          \coordinate (\nd) at (\i:2.4cm);
        \foreach \i/\ndi in {1/v1,3/v3}
        {
            \foreach \j/\ndj in {2/v2,4/v4}
            {
                \coordinate (midpoint\ndi\ndj) at ($(\ndi)!0.5!(\ndj)$);
                \ifthenelse{\i < \j}{
                \path[plusmarks] (\ndi)--(midpoint\ndi\ndj);
                \path[minusmarks] (\ndj)--(midpoint\ndi\ndj);
                }{\path[plusmarks] (\ndj)--(midpoint\ndi\ndj); 
                \path[minusmarks] (\ndi)--(midpoint\ndi\ndj);}
            }

            \foreach \j in {1,2}
            {
                \coordinate (midpoint\ndi u2\j) at ($(\ndi)!0.5!(u2\j)$);
                \path[plusmarks] (\ndi)--(midpoint\ndi u2\j);
                \path[minusmarks] (u2\j)--(midpoint\ndi u2\j);
            }
        }

        \foreach \i/\ndi in {2/v2,4/v4}
        {
            \foreach \j in {1,2}
            {
                \coordinate (midpoint\ndi u1\j) at ($(\ndi)!0.5!(u1\j)$);
                \path[plusmarks] (\ndi)--(midpoint\ndi u1\j);
                \path[minusmarks] (u1\j)--(midpoint\ndi u1\j);
            }
        }

        \foreach \i in {1,2}
        {
            \foreach \j in {1,2}
            \path[plusmarks] (u1\i)--(u2\j);
        }
          \foreach \i/\lab/\nd in {22/v_1/v1,67/v_2/v2,112/v_3/v3,157/v_4/v4,202/u_1/u11,247/u_1/u12,292/u_2/u21,337/u_2/u22}
            \node[circle, draw, fill=white, inner sep=2pt]at (\nd) {$\lab$};
        \end{tikzpicture}
    \end{center}
    \caption{The local digraph $G_2$}
    \label{fig:G2}
    \end{figure}

    Now let us count the number of induced $LDP_3$ in $G_k$.
    For each vertex, we count the number of induced $LDP_3$ using it as the middle vertex.
    For each $y\in[2k]$, it is clear that $v_xv_yv_z$ is an induced $LDP_3$ for each $x<y<z$ with $2\nmid x-y, 2\nmid z-y$.
    Moreover, it is also clear that $v_xv_yu_z^{(i)}$ is an induced $LDP_3$ for each $x<y$ and $z\in\{1,2\}$ with $2\nmid x-y, 2\nmid z-y$ and $i\in[k]$.
    This gives
    \[\left\lfloor\frac{y}{2}\right\rfloor\left(\left\lfloor\frac{2k+1-y}{2}\right\rfloor+k\right)\]
    induced $LDP_3$ using $v_y$ as the middle vertex.
    Now for each $y\in\{1,2\}$ and $i\in[k]$, we see that $v_xu^{(i)}_yu^{(j)}_z$ is an induced $LDP_3$ as well for each $x$ and $z$ with  $2\nmid x-y, 2\nmid z-y$, $z\in\{1,2\}$ and $j\in[k]$.
    This gives $k^2$ induced $LDP_3$.
    In total, we get
    \[\sum_{y=1}^{2k}\left\lfloor\frac{y}{2}\right\rfloor\left(\left\lfloor\frac{2k+1-y}{2}\right\rfloor+k\right)+2k\cdot k^2 = k\sum_{y=1}^{2k}y-\frac{1}{4}\sum_{y=1}^{2k}y^2+2k^3+O(k^2)=\frac{10}{3}k^3+O(k^2).\]
    Therefore, as $k$ tends to infinity, the fraction of triplets that induce $LDP_3$ in $G_k$ tends to 
    \[\frac{10\cdot 6}{3\cdot 4^3}= \frac{5}{16}\]
    as $k$ goes to infinity.
    
    Now we can do an iterated blowup as follows: for each $t$, starting with $H_{4^0}=G_1$, set $H_{4^t}$ to be $G_{4^t}$ union one copy of $H_{4^{t-1}}$ on the copies of $u_1$ and another copy of $H_{4^{t-1}}$ on the copies of $u_2$.
    By \cref{lem:LDAG-blowup} and a simple induction, we know that $H_{4^t}$ is an LDAG for every $t$.
    Moreover, whenever we insert $H_{4^{t-1}}$ in $G_{4^t}$, we do not lost any induced copies of $LDP_3$ in $G_{4^t}$, as no induced copies of $LDP_3$ use two copies of $u_1$ or two copies of $u_2$.
    The fraction of triplets that induce $LDP_3$ in $H_{4^t}$ increases to
    \[\sum_{k=0}^{\infty}\frac{5}{16}\left(\frac{2}{4^3}\right)^k = \frac{10}{31}\]
    as $t$ goes to infinity, showing that $\aind(LDP_3)\geq \frac{10}{31}$, as desired.
\end{proof}

\section{Small graphs}\label{section: small graphs}
In this section, we study the edge inducibilities of small graphs including graphs on at most $4$ vertices and cycles on five vertices.
\subsection{Graphs on at most $4$ vertices} \label{section: graphs smaller than 4}

In the vertex inducibility problem, there has been plenty of interest in studying it for small graphs \cite{EveLin2015, Hir2014}. With the exception of $P_4$, the vertex inducibility of all other graphs on at most 4 vertices is known \cite{EveLin2015}. Here we can similarly solve the edge inducibility of all graphs on at most 4 vertices, for which we split into the following categories as shown in \cref{tab:small-summary}.
\begin{table}[htbp]
  \centering
  \caption{$\eind$ of graphs on at most 4 vertices.}
  \label{tab:small-summary}
\resizebox{\textwidth}{!}{%
\begin{tabular}{L{2.4cm} C{4cm} C{2.0cm} L{4.4cm} C{4cm}}
\toprule
\textbf{$G$} & \textbf{Picture of $G$} & \textbf{$\eind(G)$} & \textbf{Extremal host graph $H$} & \textbf{Picture of $H$} \\
\midrule
$K_{1,2},\, K_{1,3}$
  & \GpThree\; and  \GstarKoneThree 
  & $\frac{1}{4}$, \, $\frac{1}{8}$& Star &  \parbox[c]{\linewidth}{\centering
  \GstarKoneSix} \newline \\

$C_4$
  & \GcFour
  & $\frac{1}{2}$ & Balanced complete bipartite graph &   \parbox[c]{\linewidth}{\centering
  \GkFourFour} \newline \\

$e+e$
  & \Gedgeedge
  & 1 & Matching &   \parbox[c]{\linewidth}{\centering
  \GmatchingFour} \newline \\

$e,\, K_3,\, K_4$
  & \Gedge\; and  \GkThree\; and \GkFour
  & 1 & Complete graph & \parbox[c]{\linewidth}{\centering
  \GkSix} \newline \\

$P_4$
  & \GpFour
  & $\frac{1}{4}$ & Construction by local digraph & \parbox[c]{\linewidth}{\centering
  \GkSixWithSixPendants} \newline \\

$K_3^+$
  & \GkThreePlusE
  & $\frac{1}{4}$ & Construction by local digraph & \parbox[c]{\linewidth}{\centering
  \GTrianglePendant} \newline \\

$K_4^-$
  & \GkFourMinusE
  & $\frac{1}{4}$ & Book graph, complete bipartite graph with one part being a clique ($K_a \lor \overline{K_b}$, with $b \gg a$) & \GkFiveFiveSplit \\
\bottomrule
\end{tabular}}
\end{table}
\begin{remark}
    Note that edge inducibility only makes sense for graphs with no isolated vertices; in fact, the edge inducilibity of a graph with isolated vertices is $\infty$.
    For this reason, in this section, we assume that all graphs $G$ have no isolated vertices.
\end{remark}

\begin{theorem}
    The values of $\eind(G)$ for the graphs $G$ with at most 4 vertices are as in Table~\ref{tab:small-summary}. 
\end{theorem}

\begin{proof} 
The following claim will be helpful in upper bounding the edge inducibility of the various small graphs. 
\begin{claim}\label{claim:upper-bound}
    Let $G$ be a graph containing $k$ distinct perfect matchings. Then \[\eind(G,m) \leq \binom{m}{\frac{v(G)}{2}} \cdot \frac{1}{k}.\]
\end{claim}
\begin{proof}
    Note that given a matching of size $v(G)/2$, its vertices could span at most a single induced copy of $G$. Furthermore, since $G$ contains $k$ distinct perfect matchings, it follows by double counting that for any host graph $H$ with $m$ edges, we have 
    \[ k N_{\ind}(G,H) \leq \binom{m}{\frac{v(G)}{2}},\]
    which plugging into the definition of $\eind(G,m)$ gives the desired conclusion.
\end{proof}
    \begin{enumerate}[leftmargin=*,topsep=0pt]
        \item [(1)] ($K_{1,2}$ and $K_{1,3}$) In general, we claim that $\eind(K_{1,t}) = 2^{-t}$ so when $t= 2$ we get $\eind(K_{1,2}) = \frac{1}{4}$\linebreak and when $t = 3$ we get $\eind(K_{1,3}) = \frac{1}{8}$. For the lower bound, when the host graph is the star $K_{1,m}$, we note that $N_{\ind}(K_{1,t}) = \binom{m}{t}$ and consequently, $\eind(K_{1,t}) \geq \lim_{m \to \infty}\frac{t!}{(2 m)^{t}}\binom{m}{t} = 2^{-t}$. For the upper bound, the number of induced $t$-edge subgraphs of an $m$-edge graph is at most $\binom{m}{t}$. In particular, $\eind(K_{1,t}) \leq \limsup_{m\to \infty} \frac{t!}{(2m)^{t}} \binom{m}{t}=2^{-t}$.
        
        \item [(2)] ($C_4$) While this was basically also studied by Bollob\'as, Nara and Tachibana \cite{BolNarTac1986}, which computed $\eind(K_{t,t},m)$ quite precisely, here we give a simpler proof that suffices for our purposes of showing that $\eind(K_{t,t}) = 2^{1-t}$. When $t = 2$ this recovers the special case $\eind(C_4) = \frac{1}{2}$. Recall that $\abs{\Aut(K_{t,t})}=2(t!)^2$ and that $\alpha^*(K_{t,t})=t$. For the lower bound, considering the host graph $K_{a,b}$ with $ab = m$ and $a,b \to \infty$ as $m\to\infty$, we get 
        \[
            \eind (K_{t,t})\geq \limsup_{m\to \infty}\frac{2(t!)^2}{(2m)^t} \eind(K_{t,t},m) \ge 
            \lim_{m\to \infty}\frac{2(t!)^2}{(2m)^t}\binom{a}{t}\binom{b}{t} = 2^{1-t}. 
        \]
    
        For the upper bound, since $K_{t,t}$ contains $t!$ distinct perfect matchings we have by Claim~\ref{claim:upper-bound} that $\eind(K_{t,t}) \le \lim_{m\to \infty} \frac{2(t!)^2}{t!(2m)^{t}}\binom{m}{t} = 2^{1-t}$.
    \item [(3)] ($e+e$) By taking the host graph $H$ as a large matching, it is clear that for every $t$ if $G$ is a matching on $t$ edges then $\eind(G) = 1$.  
    \item [(4)] ($K_2, \,  K_3,\, K_4$) The upper bound in  Theorem \ref{thm: Alon} is clearly achieved for cliques, and therefore our choice of normalization in $\eind(\cdot)$ gives $\eind(K_t) = 1$ for all $t$. 
    \item [(5)] ($P_4$) The case $P_4$ was already done in \cref{ex:P4}, but we will now use terminologies in \cref{section: local digraph} and apply \cref{thm:aind-eind-ind} to demonstrate how local digraphs can be useful.
    Let $P_4 = e_1- e_2 - e_3$ so that there is a unique perfect matching $M=\{e_1,e_3\}$ in $P_4$, and $\ldg(P_4, M) = e$.
    See \cref{fig:path-four} for reference.
    We have $\eind(P_4) = \frac{\ind(e)}{4} = \frac{1}{4}$ by Theorem \ref{thm:aind-eind-ind} since $\ldg(P_4, M) = e$ is undirected which clearly implies that $\aind(\ldg(P_4,M)) = \ind(\ldg(P_4, M)) = 1$.

    \begin{figure}[!htp]
        \centering
        \scalebox{0.65}{
\begin{tikzpicture}[x=1cm,y=1cm,
  plusmarks/.style={
    postaction=decorate,
    decoration = {
            markings,
            mark = 
                between positions 0.06 and 0.94 step 11.8pt 
                with
                {
                    \draw (0pt, 4pt) -- (0pt, -4pt);
                    \draw (-4pt, 0pt) -- (4pt, 0pt);
                }
            }
  }]

  \node[vtxx] (A) at (0, 1.25) {$+$};
  \node[vtxx] (B) at (0, -1.25) {$-$};
  \node[vtxx] (C) at (2.5,1.25) {$+$};
  \node[vtxx] (D) at (2.5,-1.25) {$-$};

  \draw[blackedge] (A)--(C);
  \draw[rededge]   (A)--(B) node[midway,elab,left]{\Large 1};
  \draw[rededge]   (C)--(D) node[midway,elab,right]{\Large 3};

  \begin{scope}[xshift=6.8cm]
    \node[num] (E1) at (0,0) {\large 1};
    \node[num] (E2) at (4.2,0) {\large 3};
    \path[plusmarks](E1)--(E2);
  \end{scope}
\end{tikzpicture}
}

        \caption{Local digraph for $P_4$.}
        \label{fig:path-four}
    \end{figure}
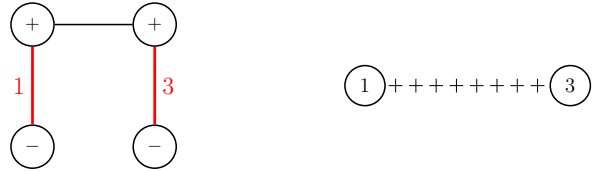
    
    \item [(6)] ($K_3^+$) Note that for the triangle with a pendant edge $K_3^+$, by taking the pendant edge and the edge of the triangle not incident to it, we get a perfect matching that is the unique fractional perfect matching (see \cref{fig:tri-pendant} below). Since $\abs{\Aut(K_3^+)}=2$ and since $\alpha^*(K_3^+)=2$, \cref{thm:aind-eind-ind} implies that $4 \eind(K_3^+) = \ind(\ldg(G,M)) = \aind(\ldg(G,M)) = 1$. The last two equalities follows from taking the host local digraph $H$ where every ordered pair of vertices $i < j$ are connected by two local edges with the sign patterns as in \cref{fig:tri-pendant}. Note that every pair of $(i,j)$ induces a $\ldg(G,M)$. Furthermore, because this host graph is acyclic, we therefore have $\aind(\ldg(K_3^+,M)) = \ind(\ldg(K_3^+,M))=1$. This gives $\eind(K_3^+) = \frac{1}{4}$, as desired. 

    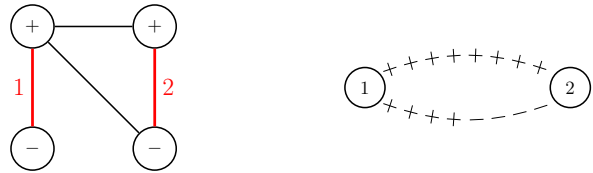
\begin{figure}[!htp]
        \centering
    
        \scalebox{0.65}{\begin{tikzpicture}[x=1cm,y=1cm, plusmarks/.style={
    postaction=decorate,
    decoration = {
            markings,
            mark = 
                between positions 0.06 and 0.94 step 12.7pt 
                with
                {
                    \draw (0pt, 4pt) -- (0pt, -4pt);
                    \draw (-4pt, 0pt) -- (4pt, 0pt);
                }
            }
  }, 
  halfplus/.style={
    postaction=decorate,
    decoration = {
            markings,
            mark = 
                between positions 0.06 and 0.48 step 12.7pt 
                with
                {
                    \draw (0pt, 4pt) -- (0pt, -4pt);
                    \draw (-4pt, 0pt) -- (4pt, 0pt);
                }
            }
  }, 
  halfminus/.style={
    postaction=decorate,
    decoration = {
            markings,
            mark = 
                between positions 0.55 and 0.94 step 12.7pt 
                with
                {
                    \draw (-4pt, 0pt) -- (4pt, 0pt);
                }
            }
  }]

  \node[vtxx] (A) at (0, 1.25) {$+$};
  \node[vtxx] (B) at (0, -1.25) {$-$};
  \node[vtxx] (C) at (2.5,1.25) {$+$};
  \node[vtxx] (D) at (2.5,-1.25) {$-$};

  \draw[blackedge] (A)--(C);
  \draw[blackedge] (A)--(D);

  \draw[rededge]   (A)--(B) node[midway,elab,left]{\Large 1};
  \draw[rededge]   (C)--(D) node[midway,elab,right]{\Large 2};

  \begin{scope}[xshift=6.8cm]
   \node[num] (A) at (0,0) {1};
\node[num] (B) at (4.2,0) {2};
\path[plusmarks] (0.3,0.3) to[out=20, in=160] (3.9,0.3);
\path[halfplus] (0.3,-0.3) to[out=-20, in=-160] (3.9,-0.3);
\path[halfminus] (0.3,-0.3) to[out=-20, in=-160] (3.9,-0.3);
  \end{scope}
\end{tikzpicture}}
        \caption{Local digraph for $K_3^+$.}
        \label{fig:tri-pendant}
    \end{figure}
    \item [(7)] ($K_4^-$) Recall that for every $s,t$ the graph $K_s \lor \overline{K_t}=2^{-t}$ is formed by a complete bipartite graph with $s$ vertices on one side and $t$ on the other, and the $s$ vertices forming $K_s$. We will show for every $t$ that $\eind(K_t \lor \overline{K_t}) = 2^{-t}$, which in particular for $t=2$ shows that $\eind(K_4^-) =\frac{1}{4}$.
        
    Indeed, fix $a,b$ such that $ab =m$ and $b \gg a \gg 1$ as $m$ tends to infinity. 
    As described in Table~\ref{tab:small-summary}, for the lower bound, we can consider the host graph given by $H=K_a \lor \overline{K_b}$. Then, any $t$ vertices from $K_a$ together with any $t$ vertices from $\overline{K_b}$ induce a $K_t \lor \overline{K_t}$, and so $N_{\ind}(K_t \lor \overline{K_t}, H) = \binom{a}{t}\binom{b}{t}$. In particular, since $\alpha^*(K_t \lor \overline{K_t}) = t$ and $\abs{\Aut(K_t \lor \overline{K_t})}=(t!)^2$ we have \[\eind(K_t \lor \overline{K_t}, m) \geq \limsup_{m\to \infty} \frac{(t!)^2}{\left(2 \cdot \left( \binom{a}{2} + ab \right) \right)^t}\binom{a}{t}\binom{b}{t} = 2^{-t}.\] 

    For the upper bound, since $K_t \lor \overline{K_t}$ contains $t!$ distinct perfect matchings, by Claim~\ref{claim:upper-bound}, it follows that $\eind(K_t \lor \overline{K_t}) \leq \lim_{m \to \infty} \frac{(t!)^2}{t!(2m)^t} \binom{m}{t} = 2^{-t}$.\qedhere
\end{enumerate}
\end{proof}

\subsection{The $5$-cycle}\label{sec: five cycle}

In this subsection, we prove \cref{thm:C5} via a simple application of the mixture bound, \cref{prop:mix}.
\begin{proof}[Proof of \cref{thm:C5}]

Let $H$ be a graph with $m$ edges. We sample an induced embedding of $C_5$, $(X_1,\dots,X_5)\in V(H)^5$ uniformly at random, so that $X_1X_2,X_2X_3,X_3X_4,X_4X_5,X_5X_1$ are the only edges among these five vertices. It follows that
\[\HH(X_1,\dots,X_5)=\log_2E_{\ind}(C_5,H).\]
Together with the uniform bound $\HH(X_1,X_2)\leq \log_2 (2m)$, it is now sufficient to prove 
\[
    2\HH(X_1,\dots,X_5)+\log_2 250\leq 5\HH(X_1,X_2).
\]
This is because substituting the uniform bound gives
\[2\log_2 E_{\ind}(C_5,H)+\log_2 250\leq 5\log (2m),\]
which implies $\eind(C_5)\leq \frac{1}{\sqrt{250}}$.

To upper bound $2\HH(X_1,\ldots, X_5)$, we use the chain rule and symmetry to get
\[2\HH(X_1,\dots,X_5)=2\HH(X_1,X_2)+2\HH(X_1\mid X_2,X_3)+2\HH(X_1\mid X_2,X_3,X_4)+2\HH(X_1\mid X_2,X_3,X_4,X_5).\]
To proceed, we need the following two claims to bound the conditional entropies. These two claims are both proved by applying the mixture bound, \cref{prop:mix}.

\begin{claim}\label{claim:EntropyIneq2}
    \[\HH(X_1\mid X_2,X_3,X_4,X_5)\leq \HH(X_1\mid X_2,X_3,X_5)\leq \HH(X_1\mid X_2)-\log_22.\]
\end{claim}
\begin{proof}
    The first inequality is simply dropping conditioning, so we will focus on the second inequality.
    Note that $(X_1,X_2,X_3,X_4)$ and $(X_1,X_2,X_4,X_3)$ have disjoint supports. This is because $X_2X_3$ is always an edge and $X_2X_4$ is always not an edge. By the mixture bound, we have
    \[2\cdot 2^{\HH(X_1,X_2,X_3,X_4)}\leq 2^{\HH(X_1,X_2,A,B)},\]
    where $(X_1,X_2,A,B)$ is a mixture of the two random tuples $(X_1,X_2,X_3,X_4)$ and $(X_1,X_2,X_4,X_3)$. Taking logarithms, we get
    \[\HH(X_1,X_2,X_3,X_4)+\log_22\leq \HH(X_1,X_2,A,B).\]
    Since $(X_3,X_4\mid X_1)$ is symmetric, we know that $(X_1,X_3,X_4)$ and $(X_1,A,B)$ have the same distribution. Thus, we have 
    \[\HH(X_1,X_2,A,B)\leq \HH(X_1,A,B)+\HH(X_2\mid X_1)=\HH(X_1,X_3,X_4)+\HH(X_1\mid X_2),\]
    where the first inequality follows from the definition of conditional entropy and dropping condition.
    Therefore, we may combine the inequalities above and get
    \[\HH(X_2\mid X_1,X_3,X_4)=\HH(X_1,X_2,X_3,X_4)-\HH(X_1,X_3,X_4)\leq \HH(X_1\mid X_2)-\log_22.\]
    By symmetry, we have
    \[\HH(X_1\mid X_2,X_3,X_5)=\HH(X_2\mid X_1,X_3,X_4),\]
    and this completes the proof.
\end{proof}

\begin{claim}\label{claim:EntropyIneq1}
    \[\HH(X_1\mid X_2,X_3,X_4)\leq \HH(X_1)-\log_2 5.\]
\end{claim}
\begin{proof}
    For $1\leq i\leq 5$, let $Y_i$ be a conditional independent resampling of $X_i$ conditioning on $X_{i+1},X_{i+2},X_{i+3}$, where the indices are considered modulo $5$.
    We claim that the supports of $(X_1,\dots,X_5,Y_i)$ for $1\leq i\leq 5$ are mutually disjoint. 
        
    Assume for the sake of contradiction that 
    \[(x_1,\dots,x_5,y)\in \supp(X_1,\dots,X_5,Y_i)\cap\supp(X_1,\dots,X_5,Y_j).\]
    By symmetry, we may assume without loss of generality that $i=1$ and $j\in\{2,3\}$. Since $(x_1,\dots,x_5,y)\in \supp(X_1,\dots,X_5,Y_1)$, we know that $yx_2$ is an edge but $yx_3,yx_4$ are not edges. If $j=2$, then we know that $yx_3$ must be an edge since $(x_1,\dots,x_5,y)\in \supp(X_1,\dots,X_5,Y_2)$, which is impossible. If $j=3$, then $yx_4$ must be an edge since $(x_1,\dots,x_5,y)\in \supp(X_1,\dots,X_5,Y_3)$, which is also impossible.

    Therefore, we may apply the mixture bound and get 
    \[\sum_{i=1}^5 2^{\HH( X_1,\dots,X_5,Y_i)}\leq 2^{\HH( X_1,\dots,X_5,Y)},\]
    where $Y$ is a mixture of $Y_1,\dots,Y_5$. Factoring out $2^{\HH( X_1,\dots,X_5)}$ from both sides and using the chain rule, we get
    \[\sum_{i=1}^5 2^{\HH(Y_i\mid X_1,\dots,X_5)}\leq 2^{\HH(Y\mid X_1,\dots,X_5)}.\]
    Note that the marginal distributions of $Y_1,\dots,Y_5$ and $X_1$ are all the same, so we have 
    \[\HH(Y\mid X_1,\dots,X_5)\leq \HH(Y)=\HH(X_1).\]
    From the definition of $Y_i$, we know that 
    \[\HH(Y_i\mid X_1,\dots,X_5)=\HH(X_i\mid X_{i+1},X_{i+2},X_{i+3})=\HH(X_1\mid X_2,X_3,X_4).\]
    Thus, we have
    \[5\cdot 2^{\HH(X_1\mid X_2,X_3,X_4)}\leq 2^{\HH(X_1)},\]
    and the claim follows from taking logarithms on both sides.
\end{proof}

Finally, we are ready to bound $2\HH(X_1,\dots,X_5)$. 
Recall that
\[2\HH(X_1,\dots,X_5)=2\HH(X_1,X_2)+2\HH(X_1\mid X_2,X_3)+2\HH(X_1\mid X_2,X_3,X_4)+2\HH(X_1\mid X_2,X_3,X_4,X_5).\]
Using \cref{claim:EntropyIneq2} and \cref{claim:EntropyIneq1}, we get
\[\HH(X_1\mid X_2,X_3,X_4,X_5)\leq \HH(X_1\mid X_2)-\log_22,\]
and
\[2\HH(X_1\mid X_2,X_3,X_4)+\HH(X_1\mid X_2,X_3,X_4,X_5)\leq 3\HH(X_1\mid X_2,X_3,X_4)\leq 3\HH(X_1)-3\log_2 5.\]
In addition, we may simply drop conditioning and get
\[2\HH(X_1\mid X_2,X_3)\leq 2\HH(X_1\mid X_2).\]
Lastly, combining the inequalities above, we get 
\begin{align*}
    2\HH(X_1,\dots,X_5)&\leq 2\HH(X_1,X_2)+3\HH(X_1)+3\HH(X_1\mid X_2)-\log_2 250\\
    &=5\HH(X_1,X_2)-\log_2 250.\qedhere
\end{align*}

\end{proof}

\section{Concluding remarks and open problems}\label{sec:conclusion}
\subsection{Locally directed graphs}
In this paper, we define locally directed graphs to generalize \cref{thm:eind-as-hard-as-ind} to \cref{thm:aind-eind-ind}, which turns out to be useful in understanding edge inducibilities of many graphs.
We remark that there seem to be many interesting problems one can ask about locally directed graphs that may or may not be related to edge inducibility itself.
Perhaps the most interesting connection we are aware of now is that deciding the vertex inducibility of $LDC_3$, the locally directed cycle of length $3$, is equivalent to the following conjecture by Erd\H{o}s and S\'os \cite{ErdSos82}.
\begin{conjecture}[Erd\H{o}s--S\'os]\label{conj:ErdosSos}
    Any $3$-uniform hypergraph on $n$ vertices with bipartite link at every vertex has at most $(\frac{1}{4}+o(1))\binom{n}{3}$ edges.
\end{conjecture}

A simple observation that might have not been written down formally is that this conjecture is equivalent to that $\ind(LDC_3)=\frac{1}{4}$ (this was communicated by Maya Sankar to the fourth author).
We remark that $LDC_3$ is not acyclic, and thus $\ind(LDC_3)$ does not give an upper bound on any edge inducibility of any graph via \cref{thm:aind-eind-ind}.
This suggests that the vertex inducibility problem of locally directed graphs should already be interesting on its own, and it can be extremely challenging even for small graphs.
Let us mention that if we instead work with directed host graphs, then it is a classical result (dating back to as early as \cite{KenBab40}) that the vertex inducibility of a directed cycle of length $3$ is exactly $\frac{1}{4}$.
This also shows that problems regarding locally directed graphs can be fundamentally harder than those for directed graphs.

Continuing this theme, in the proof of \cref{thm:P6}, we showed that $\ind(LDP_3)\leq \frac{2}{3}$ and $\aind(LDP_3)\geq \frac{10}{31}$.
It would also be interesting to determine $\ind(LDP_3)$ and $\aind(LDP_3)$ themselves.
The best lower bound we have for $\ind(LDP_3)$ is $\frac{2}{5}$ coming from the iterated blowup of $LDC_4$, and it is tempting to conjecture that the lower bounds are tight for both $\ind(LDP_3)$ and $\aind(LDP_3)$.
For comparisons, $\ind(P_3)=\frac{3}{4}$ was already known by Pippenger and Golumbic \cite{PipGol1975}, and for directed path $DP_3$ on $3$ vertices, Hladk\'y, Kr\'al' and Norin announced that $\ind(DP_3)=\frac{2}{5}$ in 2018 (see \cite{ChoLidPfe20}).
We may also define $\aind(DP_3)$ similarly to how we define acyclic vertex inducibility for local digraphs, where now we require the host graph to be acyclic.
In this case, we can prove that $\aind(DP_3)=\frac{1}{4}$, and a proof is provided in \cref{sec:aind-DP3}.

\subsection{Tightness of \cref{thm:aind-eind-ind}}
Let $G$ be a graph with a perfect matching $M$ that is the unique fractional perfect matching.
\cref{thm:aind-eind-ind} gives that $\frac{2^{\abs{M}}\abs{M}!}{\Aut(G)}\eind(G)\geq \aind(\ldg(G,M))$, and it is tight when $\aind(\ldg(G,M))=\ind(\ldg(G,M))$.
It would be interesting to know whether the equality always holds.

\begin{question}
    If $G$ is a graph with a perfect matching $M$ that is the unique fractional perfect matching, is it always true that
    \[\frac{2^{\abs{M}}\abs{M}!}{\Aut(G)}\eind(G)= \aind(\ldg(G,M))?\]
\end{question}

\subsection{Edge inducibilities of small graphs}
In this paper, the special cases that are not completely settled are $P_6$ and $C_5$.
We have conjectured in \cref{conj:cycle} that the edge inducibility of any cycle $C_k$ with $k\geq 4$ should match the one given by balanced blowups of $C_k$.
For $P_6$, we make the following bold conjecture.

\begin{conjecture}
    The edge inducibility of $P_6$ is equal to $\frac{5}{372}$.
\end{conjecture}
For general paths, we do not have a good candidate for conjectures, which is similar to the case of the vertex inducibility of paths.

\subsection{Upper tails for induced counts}
We conclude with another motivation mentioned in the introduction. Indeed, there is a tight connection between the upper tail problem and the problem of estimating $c_{\mathcal{H}}(G)$ for $\mathcal{H}$ being the family of $n$-vertex and $m$-edge graphs. As was shown by the second author \cite{Coh2024}, studying $\ind_{\mathcal{H}}(G)$ for the same family $\mathcal{H}$ can shed light on the upper tail problem for induced counts. We believe that studying the edge inducibility problem is the first step towards studying this problem, similarly to the case of $c_{\mathcal{H}}$ \cite{Alo1981, FriKah1998, JanOleRuc2004}.

Let us highlight that the breakthrough paper of Harel, Mousset, and Samotij \cite{HarMouSam2022} has two major steps towards applying a union bound type argument. The first is to gain control on the upper tail probability via the existence of ``seeds''. The second is to analyze the seed structure. The first step in \cite{HarMouSam2022} is not available for induced counts; that said, following their suggestion, Theorem 2.1 in \cite{Coh2024} extends their first step to induced counts. This leaves the main task of the second step. In both works, evaluating $c_{\mathcal{H}}(G)$ and $\ind_{\mathcal{H}}(G)$ are central tools in analyzing seeds and counting a compressed variant of them, enabling the use of the union bound.

\section{Acknowledgement}
We wish to thank Veronica Bitonti for numerous helpful discussions.
We would also like to thank the organizers of the Park City Mathematics Institute 2025 and Random Structures and Algorithms 2025 for bringing us together, making the collaboration possible. 
The second author is supported in
part by ERC Consolidator Grant 863438, and in part by ERC Consolidator Grant 101044123 (RandomHypGra). The fourth author is supported by the Jane Street Graduate Research Fellowship.

\bibliographystyle{plain}
\bibliography{bib.bib}

\appendix
\section{Acyclic inducibility of directed path on three vertices}\label{sec:aind-DP3}
Let $DP_3$ be the directed path on $3$ vertices, and let
\[\aind(DP_3) \eqdef \limsup_{n\to\infty}\max_{H:\textup{ DAG with }v(H)=n}\frac{N_{\ind}(DP_3,H)}{\binom{n}{3}}.\]
In this appendix, we prove the following.
\begin{theorem}\label{thm:aind-DP3}
    $\aind(DP_3) = \frac{1}{4}.$
\end{theorem}
\begin{proof}
    Note that by taking the iterated blowup of $DP_3$ itself, we get
    \[\aind(DP_3)\geq \frac{3!}{3^3-3}=\frac{1}{4}.\]
    It remains to show the upper bound.

    Let $H$ be any DAG with $n$ vertices. 
    Take the topological sort of $v_1,v_2,\ldots, v_n$ of $H$, so that an edge between $v_i,v_j$ must point in the $v_i\to v_j$ direction if $i<j$.
    Then an induced copy of $DP_3$ corresponds to a triple $i<j<k$ with $v_i,v_k$ both adjacent to $v_j$ but $v_i,v_k$ are not adjacent.

    Now for any $i<j<k$, define $I(i,j,k),J(i,j,k),K(i,j,k)$ as follows:
    \begin{enumerate}
        \item $I(i,j,k)=1$ if in $v_iv_jv_k$, there is exactly one edge incident to $v_i$, and $0$ otherwise.
        \item $J(i,j,k)=-1$ if in $v_iv_jv_k$, there is exactly one edge incident to $v_j$, and $0$ otherwise.
        \item $K(i,j,k)=1$ if in $v_iv_jv_k$, there is exactly one edge incident to $v_k$, and $0$ otherwise.
    \end{enumerate}
    The key claim is that for any $i<j<k$, we have $I(i,j,k)+J(i,j,k)+K(i,j,k)\geq 0$, and furthermore, it is $2$ if $v_i,v_j,v_k$ forms an indirect triplet.
    To check that it is always at least $0$, we just need to make sure that when there is exactly one edge incident to $v_j$, the same must hold for one of $v_i$ or $v_k$, which clearly holds.    
    It remains to check $I(i,j,k)+J(i,j,k)+K(i,j,k)=2$ when $v_i,v_j,v_k$ form an indirect triplet.
    This is clear as $I(i,j,k)=K(i,j,k)=1$ and $J(i,j,k)=0$ in this case.
    
    As an immediate consequence,
    \[N_{\ind}(DP_3,H)\leq \frac{1}{2}\sum_{1\leq i<j<k\leq n}\left(I(i,j,k)+J(i,j,k)+K(i,j,k)\right).\]
    To evaluate this, for each $s\in [n]$, let $f_s$ be the number of edges going forward from $v_s$, i.e. $f_s = \abs{\left\{i\in [s-1]: v_iv_s\in E(H)\right\}}$.
    Similarly, let $b_s$ be the number of edges going backward from $s$, which is simply $\abs{\left\{j\in [n]\backslash [s]: v_sv_j\in E(H)\right\}}$.
    Then 
    \[\sum_{i<j<k}I(i,j,k) = \sum_{i=1}^{n}b_i(n-i-b_i)\]
    and similarly
    \[\sum_{i<j<k}K(i,j,k) = \sum_{k=1}^{n}f_k(k-1-f_k).\]
    Finally, we have
    \[\sum_{i<j<k}J(i,j,k) = -\sum_{j=1}^{n}\left(f_j(n-j-b_j)+b_j(j-1-f_j)\right).\]
    
    Adding those up, we get
    \[\sum_{i<j<k}I(i,j,k)+J(i,j,k)+K(i,j,k) = \sum_{s=1}^{n}(b_s-f_s)\left(n+1-2s-(b_s-f_s)\right)\leq \frac{1}{4}\sum_{s=1}^{n}\left(n+1-2s\right)^2.\]
    Therefore 
    \[N_{\ind}(DP_3,H)\leq \frac{1}{8}\sum_{s=1}^{n}(n+1-2s)^2 = \frac{1}{8}\left(\frac{1}{3}n^3+O(n^2)\right) = \left(\frac{1}{24}+o(1)\right)n^3.\]
    This shows that $\ind(DP_3)\leq \frac{3!}{24}=\frac{1}{4}$, as desired.
    \end{proof}
\end{document}